\documentclass[amssymb, 11pt]{amsart}

\usepackage{latexsym}

\newcommand{\C}{\mathcal{C}}
\newcommand{\Z}{\mathbb{Z}}
\newcommand{\F}{\mathbb{F}}
\newcommand{\Inndiag}{\mathrm {Inndiag}}

\def\skipa{\vspace{-1.5mm} & \vspace{-1.5mm} & \vspace{-1.5mm}\\}

\newlength{\standardunitlength}
\newtheorem{prop}{Proposition}[section]

\newtheorem{lemma}[prop]{Lemma}
\newtheorem{cor}[prop]{Corollary}

\newtheorem{theorem}[prop]{Theorem}

\newcommand{\aut}{\mathrm Aut}

\begin{document}

\begin{center} \title [Bounds on conjugacy classes] {\bf Bounds on the
number and sizes of conjugacy classes in finite Chevalley groups with
Applications to Derangements}
\end{center}

\author{Jason Fulman}
\address{University of Southern California \\  Los Angeles, CA 90089-2532}
\email{fulman@usc.edu}

\author{Robert Guralnick}
\address{University of Southern California \\
Los Angeles, CA 90089-2532}
\email{guralnic@usc.edu}

\keywords{number of conjugacy classes, simple group, Chevalley groups,
partition, derangements, generating function}

\subjclass[2000]{20G40, 20B15}

\date{January 23, 2009}

\thanks{Fulman was partially supported by National Science Foundation grants DMS
0503901, DMS 0802082, and National Security Agency grants MDA904-03-1-004,
H98230-08-1-0133. Guralnick was partially supported by National Science
Foundation grants DMS 0140578 and DMS 0653873}

\begin{abstract} We present explicit upper bounds for the number and
size of conjugacy classes in finite Chevalley groups and their variations.
These results
have been used by many authors to study zeta functions associated to
representations
of finite simple groups,  random walks on Chevalley groups,
 the final solution to
the Ore conjecture about commutators in finite simple groups
and other similar problems.
In this paper, we solve a strong version of the Boston-Shalev conjecture on
derangements in simple groups for most of the families of
primitive permutation
group representations of finite simple groups
 (the remaining cases are settled in
two other papers of the authors and applications are given in a third).
\end{abstract}

\maketitle

\section{Introduction} \label{intro}

One might expect that there is nothing more to be done with the study of
conjugacy classes of finite Chevalley groups. For instance over forty
years ago Wall \cite{W} determined the conjugacy classes and their sizes
for the unitary, symplectic, and orthogonal groups.  However the
formulas involved are complicated and it is not automatic to derive
upper bounds on numbers of classes or their sizes. Moreover many
applications seem to require such bounds (in particular, universal explicit
bounds of the form $cq^r$ where $r$ is the rank of the ambient algebraic
group and $q$ is the size of the field of definition). To convince the reader of
this, we mention some places in the literature where bounds on the
number of conjugacy classes in finite classical groups  were
needed:

\begin{enumerate}
\item The work of Gluck \cite{Gl} on convergence rates
of random walks on finite classical groups. His bounds were of the form
$c q^{3r}$.
\item The work of Liebeck and Pyber \cite{LP} on number of
conjugacy classes in arbitrary groups; for finite groups of Lie type their
bound was $(6q)^r$.
\item
The work of Maslen and Rockmore \cite{MR} on computations of Fourier
transforms; they obtained a bound of $q^n$ for $GL(n,q)$ and $8.26 q^n$
for $U(n,q)$. These bounds are of the type we prove here, namely $c q^r$
where $c$ is explicit.
\item Liebeck and Shalev \cite{LS1} have used bounds in the current paper
to study probabilistic results about homomorphisms of certain Fuchsian
groups into Chevalley groups, and random walks on Chevalley groups
\cite{LS2}. Shalev used these results in a crucial way to study the images
of word maps \cite{Sh}. Our bounds were also critical in the solution of
the Ore conjecture on commutators in finite simple groups \cite{LOST}.
\item Our results have been used in studying various versions of Brauer's
$k(GV)$ problem  \cite{GT} -- in particular, the noncomprime version and some new
related conjectures of Geoff Robinson   \cite{R}.  They were used
in \cite{GR} in obtaining new results about the commuting probability
in finite groups.  In particular,
these results will be useful in improving results of Liebeck-Pyber \cite{LP}
and Mar\'oti \cite{Ma} about the number of conjugacy classes in completely
reducible linear groups over finite fields and in permutation groups.
\end{enumerate}

We also use our results to prove a large part
 of  a conjecture of Boston et. al.
\cite{Bo} and Shalev stating that the proportion of fixed point free
elements of a finite simple group in a transitive action on a finite set
$X$ with $|X|>1$ is bounded away from zero.   This immediately reduces to
the case of primitive actions (and so to studying maximal subgroups of
simple groups).   This conjecture has applications to random
generation of groups \cite{FG4} and to maps between varieties over finite
fields \cite{GW}.  In fact, we prove much stronger results for many
actions of almost simple groups in this paper as a consequence of our
bounds on class numbers and centralizer sizes.  See \cite{FG1} where the
bounded rank case was handled. The remaining cases are treated in
\cite{FG2, FG3}.

We now state some of the main results of the paper.  If $G$ is a finite group,
we let $k(G)$ denote the number of conjugacy classes of $G$.   See
\cite{GLS3} for background on Chevalley groups.

\begin{theorem} \label{A}  Let $G$ be a connected
simple algebraic group of rank $r$ over a field of positive
characteristic. Let $F$ be a Steinberg-Lang endomorphism of $G$ with $G^F$
a finite Chevalley group over the field  $\F_q$.
\begin{enumerate}
\item   $q^r < k(G^F) \le 27.2q^r$.   
\item   $k(G^F) \le q^r + 68q^{r-1}$.  In particular,          
$lim_{q \rightarrow \infty} k(G^F)/q^r =1$ and the convergence is uniform
with respect to $r$.
\item The number of conjugacy classes of $G^F$ that are not semisimple
is at most $68q^{r-1}$.
\item The number of conjugacy classes of $G$ that are $F$-stable
is between $q^r$ and $27.2 q^r$.
\end{enumerate}
\end{theorem}

There are much better bounds on the constants in Theorem \ref{A} for many
of the families. The correct upper bound for part 1 is about $15.2 q^r$
(for $Sp(2n,2)$).   We give limiting values for $k(G^F)/{q^r}$ (as $r
\rightarrow \infty$ with $q$ fixed) for each of the  families of classical
groups.  In particular, we see that this ratio does not tend  to $1$ for
$q$ fixed. See the tables  in Section \ref{exceptional} for a summary of
the results. There are precise formulas for the number of conjugacy
classes for the exceptional Chevalley groups -- see \cite{Lu} (and
similarly, one can work out such formulas for the low rank classical
groups).   As we have already noted earlier in the introduction, the
existence of a $C(q)$ that depends on $r$ has already been proved (and
this is straightforward). In fact, using the results of Lusztig and others
about unipotent classes, it is easy to prove that $k(G^F)$ is bounded
above by a monic polynomial in $q$ of degree $r$ (independently of $q$),
whence for a fixed $r$, it follows  that $k(G^F)/q^r \rightarrow 1$
as $q \rightarrow \infty$.  One of the  key features of our result is that our
bounds are independent of the rank of the group, and many of the applications
depend on this.

We show that one can get a similar bound for almost simple
Chevalley groups allowing all types of outer automorphisms.

\begin{cor} \label{B} Let $G$ be an almost simple
group with socle $S$, a Chevalley group of rank $r$ defined over
$\F_q$.
\begin{enumerate}
\item $k(S) \le 15.2 q^r$.
\item  $k(S) \le q^r + 30q^{r-1}$.
\item  $k(G) \le  100 q^r$.
\end{enumerate}
\end{cor}

Recall that a permutation is called a derangement if it has no fixed
points. If $G$ is a transitive group on a finite set $\Omega$, define
$\delta(G,\Omega)$ to be proportion of derangements in $G$. By an old
theorem of Jordan, it follows that $\delta(G,\Omega) > 0$ if $|\Omega|>1$.
By a much more recent (but still elementary) theorem \cite{CC},
$\delta(G,\Omega) \ge 1/|\Omega|$ for $|\Omega| > 1$.  See Serre \cite{Se}
for many applications of Jordan's theorem.  See \cite{GW} for applications
of better bounds of $\delta(G,\Omega)$ and bounds on derangements in a
given coset.  If $\Omega$ is the coset space $G/H$, we write
$\delta(G,\Omega)$ as $\delta(G,H)$.

\begin{theorem} \label{C} Let $G$ be a classical Chevalley group
defined over $\F_q$ of rank $r$. Let $H$ be a maximal subgroup of $G$ that
acts irreducibly and primitively on the natural module and  does not
preserve an extension field structure. Then there is a universal constant $\delta >
0$ such that $\delta(G,H) > \delta$.  Moreover, $\delta(G,H) \rightarrow
1$ as $r \rightarrow \infty$.
\end{theorem}

The case where $r$   is bounded (including the case of exceptional groups)
was dealt with in \cite{FG1}.   In that case, the first statement
about the existence of $\delta$ is the same.  The second statement is
valid if and only if $H$ does not contain a maximal torus.
We give another proof here.  In fact, we
prove a much stronger result than Theorem \ref{C}. See Theorems \ref{bstronger} and \ref{stronger}.
We also prove some results about derangements in cosets of simple groups.
See Section \ref{boston-shalev}.

The three remaining families of maximal subgroups (reducible subgroups --
including parabolic subgroups, groups preserving an extension field
structure and imprimitive groups) are dealt with in \cite{FG2, FG3}.

The Boston-Shalev conjecture was proved for alternating and symmetric
groups in \cite{LuP}.   See also \cite{D}, \cite{DFG}, and \cite{FG3}.

Two other results of interest that we prove and use in the
preceding result are:

\begin{theorem}  \label{DD}  Let $G$ be a connected
simple algebraic group of rank $r$  of adjoint type over a field of positive
characteristic. Let $F$ be a Steinberg-Lang endomorphism of $G$ with $G^F$
a finite Chevalley group over the field $\F_q$.   There is an absolute constant $A$ such
that for all $x \in G^F$,
 $$
 |C_{G^F}(x)| >  \frac{q^{r}}{A (1 + \log_q r)}.
 $$
\end{theorem}

See Section \ref{mincent} for bounds of the form in Theorem \ref{DD}, with
explicit constants in all cases.
The result holds for the finite simple Chevalley
groups as well except that if $G = PSL(n,q)$ or $PSU(n,q)$,    then $q^{n-1}$
needs to be replaced by $q^{n-2}$.   The result also holds for orthogonal
groups except that in even dimension, $q^r$ needs to be replaced
by $2q^{r-1}$ (but only for elements outside $SO$).

\begin{theorem}  \label{E} Let $G$ be a finite simple Chevalley group
defined over $\F_q$ of rank $r$ with $q$ a power of the prime $p$.
The number of conjugacy classes of maximal subgroups of
$G$ is at most  $Ar(r + r^{1/2}p^{3r^{1/2}} + \log \log q)$
for some constant $A$.
\end{theorem}

In fact, we conjecture that the $r^{1/2}p^{3r^{1/2}}$ term can be removed
above.    If $r$ is bounded,  a much stronger result is given in
\cite{LMS}.

The organization of this paper is as follows. Section \ref{outer} studies
the number of conjugacy classes in a given coset of a normal subgroup.
In particular, we give a very short proof (Lemma \ref{weak shintani})
of a generalization of results in \cite{BW, I}  on the distribution of conjugacy
classes over cosets of some normal subgroup.
Section \ref{boundnumber} obtains explicit and sharp upper bounds and
asymptotics for the number of conjugacy classes in finite classical groups
(some of these were announced in the survey \cite{FG1}). This mostly
involves a careful analysis of Wall's generating functions for class
numbers, but we do obtain new generating functions for groups such as
$SO^{\pm}(2n,q)$ and $\Omega^{\pm}(2n,q)$  with $q$ odd.
 Section \ref{exceptional} tabulates some of the results
from previous sections and summarizes corresponding results for
exceptional Chevalley groups, due to L\"ubeck \cite{Lu} and others. In
Section \ref{class results}, we turn to almost simple groups, proving
Theorem \ref{A}, Corollary \ref{B}, and some related results.
 Section \ref{mincent} derives explicit lower bounds on centralizer
sizes (and so  upper bounds on the sizes of conjugacy classes) in finite
classical groups.
In Section \ref{boston-shalev}, we get upper bounds for the number of
conjugacy classes in a maximal subgroup aside from three families of
maximal subgroups.  We then combine those results with Theorem \ref{D} to
obtain Theorem \ref{C}, Theorem \ref{stronger}, and related results on
derangements.

\section{Outer Automorphisms} \label{outer}

In this section, we prove some results about the number of conjugacy
classes in a given coset.   This will allow us to pass between
various forms of our group.
We first recall an elementary result of  Gallagher \cite{Ga}.

\begin{lemma} \label{Boblemma}  Let $H, N$ be subgroups of the finite
group $G$ with $N$ normal in $G$. Then
\begin{enumerate}
\item $|G:H|^{-1}k(H) \le k(G) \le |G:H|k(H)$; and
\item $k(G) \le k(N)k(G/N)$.
\end{enumerate}
\end{lemma}

The following lemma will be useful in getting some better bounds
for almost simple groups.
 See also \cite{BW, I, K} for similar but somewhat weaker results.
 Many of the proofs of related results use character (or Brauer
 character) theory (in particular, Brauer's Lemma), and thus do
 not immediately extend to the case of $\pi$-elements.

 Our method of proof is entirely different -- it is shorter, more elementary
 and based on a very easy variant of what is known as Burnside's Lemma.

\begin{lemma} \label{weak shintani} Let $N$ be a normal subgroup
 of the finite group $G$ with $G/N$ cyclic and generated by $aN$.
 Let $\pi$ be a set of primes containing all prime divisors of $|G/N|$.
 Set $\alpha$ to be the number of $G$-invariant conjugacy classes
 of $\pi$-elements of $N$.
 \begin{enumerate}
 \item  The number of $G$-conjugacy classes of $\pi$-elements
 in the coset $bN$
 that are a single $N$-orbit is equal to $\alpha$ for any coset $bN$.
 \item The number of $G$ conjugacy classes of $\pi$-elements in the coset $aN$
 is $\alpha$.
 \end{enumerate}
 \end{lemma}

\begin{proof}    Note that $G$ acts on the coset $bN$ by conjugation,
and thus on the $\pi$-elements in that coset.
We want to calculate the number of common $G, N$ orbits on
the $\pi$-elements of $bN$.

By a slight variation of Burnside's Lemma (with essentially the same proof
--- see \cite[\S13]{FGS}), this is the average number of fixed points of an
element in the coset $aN$.   Let $x \in aN$.  It follows that
$C_G(x) \cap a^iN =  x^iC_N(x)$.    Let $y=x^e$ where
$e \equiv 1 \mod [G:N]$ and $y$ is a $\pi$-element (this is possible
since $\pi$ contains all prime divisors of $[G:N]$).  Then $y^iN=x^iN$
for all $i$, and so $y^iC_N(x)=x^iC_N(x)$ for all $i$.
 If $w \in C_N(x) \le C_N(y)$, then $y^iw$ is a $\pi$-element
if and only if $w$ is.  Thus, the number of $\pi$-elements
in $C_G(x) \cap a^iN$ is the number of $\pi$-elements in $C_N(x)$  --
in particular, this number is independent of the coset.  This proves (1).

If $x  \in aN$, then  $G=NC_G(x)$ and so
$x^N=x^G$, whence every $G$-conjugacy class in $aN$
is a single $N$-orbit.    So (1) implies (2).
\end{proof}

This allows us to prove a generalization of part of Lemma \ref{Boblemma}.
If $\pi$ is  a set of primes and $X$ is a finite group, let
$k_{\pi}(X)$ be the number of conjugacy classes of $\pi$-elements of $X$.

\begin{lemma} \label{Boblemma2}  Let $ N$ be  a normal subgroup of the finite
group $G$    Let $\pi$ be a set of primes.
Let $x_iN$ denote a set of representatives of the $\pi$-conjugacy
classes of $G/N$.  Let $f(x_i)$ denote the number of $N$-conjugacy
classes of $\pi$-elements that are $x_i$-invariant.
$$
k_{\pi} (G) \le \sum_{i=1}^r f(x_i)  \le  k_{\pi}(G/N)k_{\pi}(N).
$$
 \end{lemma}

\begin{proof}
If $y \in G$ is a $\pi$-element, then $yN$ is conjugate to $x_iN$
for some $i$.

Set $G_i = \langle  N, x_i \rangle$ and note that every prime divisor
of $[G_i:N]$ is in $\pi$.
By Lemma \ref{weak shintani}, the number of  $G_i$ conjugacy classes
of $\pi$-elements in $x_iN$  is at most $f(x_i) \le k_{\pi}(N)$.    Thus, the number of
$G$-conjugacy classes of $\pi$-elements that intersect $x_iN$
is at most $f(x_i)$.  This completes the proof.
\end{proof}

\begin{cor} \label{outer2}  Let $N$ be a normal subgroup of $G$
with $\pi$ a set of primes containing all prime divisors of $|G/N|$. If
each $\pi$-element $g$ of $G$ satisfies $G=NC_G(g)$, then $G/N$ is abelian
and the  number of $\pi$-conjugacy classes of $G$  in any coset of
$N$ is $k_{\pi}(N)$.
\end{cor}

\begin{proof}    We first show that $G/N$ is abelian.  Consider a coset
$xN$.    Let $x \in G$.  Then $xN=yN$ where $y$ is a $\pi$-element
(as in the previous proof).   Then $[y,G]=[y,NC_G(y)] \le N$.
Hence $[G,G] \le N$, and so $G/N$ is abelian.

The hypothesis implies that every $\pi$-class of $G$ is a single
$N$-orbit. Applying Lemma \ref{weak shintani} to the subgroup $H:=\langle
N, x \rangle$ shows that the number of common $H, N$-orbits on
$\pi$-elements of $xN$ is the number of common $H,N$-orbits on
$\pi$-elements in $N$. The hypothesis implies that all $G$-orbits on
$\pi$-elements are $N$-orbits, whence the number of conjugacy classes of
$\pi$-elements in $xN$ is $k_{\pi}(N)$.
\end{proof}

The hypotheses  apply to the case where $N$ is a quasisimple Chevalley
group in characteristic $p$  and $G$ is contained in the
group of inner diagonal automorphisms of $N$ with
$\pi$ consisting of all primes other than $p$.  See
\cite[2.12]{S1}.    Thus, we have:

\begin{cor} \label{same}    Let $S$ be a quasisimple
Chevalley group.  Assume that $S \le G \le \Inndiag(S)$.
Then
the number of semisimple conjugacy classes  in each coset
of $S$ in $G$ is the same.
 \end{cor}

Another easy consequence of Lemma \ref{weak shintani} is showing
that
the finiteness of unipotent classes in a disconnected reductive
group follows from the result for connected reductive groups.
There have been several proofs of the finiteness of the number
of unipotent classes in the connected case -- see \cite{Lz1}.
The result is also known for the disconnected case (see
\cite{gurint} for a generalization).

\begin{lemma}   Let $G$ be an algebraic group defined
over a finite field $L$ of characteristic $p$.   Let $H$ be its connected
component.   Suppose that $H$ has  finitely many  conjugacy classes of unipotent
elements.  Then $G$ has finitely many conjugacy classes of
unipotent elements.
\end{lemma}

\begin{proof}  Let $U$ be the variety of unipotent elements
in $G$.   So $U$ is defined over $L$.    Let $k$ be the algebraic closure
of $L$.        Suppose that $H$ has $m$ conjugacy classes of unipotent
elements. Let $L'/L$ be a finite extension. By Lang's theorem,  $H(L')$
has at most $me$ conjugacy classes of unipotent elements, where $e$ is the
maximal number of connected components in $C_H(u)$ for $u$ a unipotent
element of $H$.

Let $s$ be the number of conjugacy classes of $p$-elements in $G/H$ (which
is isomorphic to $G(L')/H(L')$). By Lemma \ref{Boblemma2},  $G(L')$ has at
most $sme$ conjugacy classes of $p$-elements.

Since $G(k)$ is the union of the $G(L')$ as $L'$ ranges over
all finite extensions of $L'/L$, it follows that $G(k)$ has at most
$sme$ conjugacy classes of unipotent elements.

By   \cite[Prop 1.1]{GLMS}), it follows that the number of unipotent
classes of $G$ is the same as the number of $G(k)$ classes of unipotent
elements..
\end{proof}

 The following must surely be known, but it follows easily
from Lemma \ref{weak shintani}.

\begin{cor}  \label{altsym}  Let $m > 3$.
Then $k(A_m) <  k(S_m)$.
\end{cor}

\begin{proof}  Let $a$ be the number
of $A_m$ classes that are stable under $S_m$.
Let $b$ be the number of $S_m$ classes in $A_m$
that are not $A_m$ classes. Clearly  $k(A_m)=a+2b$,
and by Lemma \ref{weak shintani}, $k(S_m)=2a+b$.
 So we only need to show that
$a > b$.    If $m= 4, 5$,
the result is clear.  So we show that $a > b$ for $m > 5$.

Note that $b$ is precisely the number of classes where all cycle lengths
are distinct and odd.  There clearly is an injection into stable classes;
namely since $m>4$ the largest cycle is odd of length $j \geq 5$, so one
can replace it by a product of two $2$-cycles and $j-4$ fixed points. The
image misses an element of order $4$, and so the injection is not
surjective.
\end{proof}

 Another easy consequence of Lemma \ref{weak shintani} is:

\begin{cor}  \label{outer3} Let $N$ be a normal subgroup of the finite group $G$.
Let $K$ be a  subgroup of $G$ containing $N$ with $K/N$  cyclic and
central in $G/N$.    Let $\pi$ be a set of primes containing
all prime divisors of $|K/N|$. Let $\Delta$ be the set of $G$-conjugacy classes
of $\pi$-elements
such that $K=NC_G(g)$.  Then $\Delta$ is equally distributed
among the cosets of $N$ contained in $K$.
\end{cor}

\begin{proof}   Let $\Gamma$ be the union of the conjugacy classes
in $\Delta$.  Note that $g \in \Gamma$ implies that $g \in K$.

 Let $\alpha$ be the number of
$K$-stable conjugacy classes of $\pi$-elements of $N$.

By the proof of Lemma \ref{weak shintani},
it follows that $K$ has precisely $\alpha$ orbits on $\Gamma \cap gN$
for each $g \in K$.   Since $G/K$ acts freely on the $K$ orbits on
$\Gamma$, it follows that there are precisely
$\alpha/[G:K]$ elements of $\Delta$ in each coset of $K/N$.
 \end{proof}

In certain cases,  one can describe the conjugacy classes
in a coset very nicely using
the Shintani correspondence.  See \cite[\S2]{K}.
We first need some notation.  Let $G$ be a connected algebraic group.
Let $F$ be a Lang-Steinberg endomorphism of $G$ (i.e. the fixed points
$G^F$ form a finite group).  We first recall the well known result
of Lang-Steinberg.

\begin{lemma} \label{Stein}  Let $G$ be a connected linear algebraic group,
and  let $F$ be a surjective endomorphism of $G$ such  that $G^F$ is finite.
Then the map $f: x \mapsto x^{-F}x$ from $G$ to $G$ is surjective.
\end{lemma}

Note that if $F$ is such an endomorphism of a simple connected algebraic group
$G$, then we can attach a prime power $q = q_F$ of the characteristic to $F$.
Then $G^F$ is said to be defined over $q$.  We write $G^F=G(q)$ (of course,
there may be more than one endomorphism associated with the same $q$ -- in particular,
this is the case if $G$ admits a graph automorphism).
The Shintani correspondence is:

\begin{theorem} \label{fieldaut} Let $G$ be a connected linear
algebraic group with Frobenius map $F$. Let $H$ be the fixed points of
$F^m$. We view $F$ as an automorphism of $H$ of order $m$. Then there is a
bijection $\psi$ between conjugacy classes in the coset $FH$ and conjugacy
classes in $H^F$. Moreover, $C_H(Fh) \cong C_{H^F}(k)$ where
$\psi[Fh]=[k]$.
 \end{theorem}

We may define  $\psi$ as follows.  Given $x$ in $H$,
let $\alpha_x$ be such that $\alpha_x^{-F} \alpha_x = x$.
Define $N(x)= x^{F^{m-1}} \cdots x^{F^{2}} x^{F} x$.
Set $\psi(Fx)=
 \alpha_x N(x) \alpha_x^{-1} \in H^F$.  This depended upon the choice
of $\alpha_x$, but another choice preserves the conjugacy class, and
$\psi$ defines the desired bijection on classes.   It is straightforward
to see that this bijection has the properties described in the theorem.

Combining this theorem together with Lemma \ref{weak shintani} gives:

\begin{cor} \label{shintani2} Let $G$ be a connected linear
algebraic group with Frobenius map $F$. Let $H$ be the fixed points of
$F^m$. We view $F$ as an automorphism of $H$ of order $m$.
Then $k(H^F)$ is equal to the number of $H$-conjugacy classes in
the coset $FH$ and is also equal to the number of $F$-stable conjugacy
classes in $H$.
\end{cor}

\begin{proof}  The previous theorem implies that the first two quantities
are equal.  Lemma \ref{weak shintani} implies that the second  and third
quantities are equal.
\end{proof}

\section{Number of conjugacy classes in classical groups} \label{boundnumber}

In this section, we obtain upper bounds for $k(G)$ with $G$ a classical
group. Subsection \ref{prelim} develops some preliminary tools. Type A
groups are treated in Subsections \ref{GL} and \ref{U}; symplectic and
orthogonal groups are treated in Subsections \ref{Sp} and \ref{O}
respectively.

\subsection{Preliminaries} \label{prelim}

The following result of Steinberg \cite[14.8, 14.10]{St} gives lower
bounds for $k(G)$ with $G$ a Chevalley group. See also \cite[p. 102]{C}.
The inequality is not stated explicitly though.

\begin{theorem} \label{ss} Let $G$ be a connected reductive group of semisimple rank
$r > 0$. Let $F$ be a Frobenius endomorphism of $G$ associated to $q$. Let
$Z_0$ denote the connected component of the center of $G$. The number of
semisimple conjugacy classes in $G^F$ is at least $|Z_0^F|q^r$ with equality if
$G'$ is simply connected.   In particular, $k(G^F) > q^r$.
\end{theorem}

\begin{proof}
By \cite[14.8]{St}, the number of $F$-stable conjugacy classes is
exactly $|Z_0^F|q^r$.  If $G'$ is simply connected, then the centralizer
of any semisimple element is connected, whence there is a bijection
between stable conjugacy classes of semisimple elements and semisimple
conjugacy classes in $G^F$.

On the other hand, every $F$-stable class has a representative in $G^F$
by \cite[14.10]{St}, and so  there are at least $q^r$ semisimple conjugacy
classes in $G^F$ (with equality in the simply connected case).
Since there must be at least $1$ stable class of nontrivial
unipotent elements, the last statement follows.
\end{proof}

The following two asymptotic lemmas will be useful.

\begin{lemma} \label{bign} (Darboux \cite{O}) Suppose that $f(u)$ is
analytic for $|u|<r,r>0$ and has a finite number of simple poles on
$|u|=r$. Let $w_j$ denote the poles, and suppose that $f(u)=\sum_j
\frac{g_j(u)}{1-u/w_j}$ with $g_j(u)$ analytic near $w_j$. Then the
coefficient of $u^n$ in $f(u)$ is
\[ \sum_j \frac{g_j(w_j)}{w_j^n}+o(1/r^n).\] \end{lemma}

\begin{lemma} (\cite{O}) \label{maxmodulus} Suppose that $f(u)$ is
analytic for $|u|<R$. Let $M(r)$ denote the maximum of $|f|$ restricted to
the circle $|u|=r$. Then for any $0<r<R$, the coefficient of $u^n$ in
$f(u)$ has absolute value at most $M(r)/r^n$. \end{lemma}

The following lemma is Euler's pentagonal number theorem (see for instance
page 11 of \cite{A1}).

\begin{lemma} \label{pent} For $q>1$,
\begin{eqnarray*} \prod_{n=1}^{\infty} (1-\frac{1}{q^i}) & = & 1 +
\sum_{n=1}^{\infty} (-1)^n (q^{-\frac{n(3n-1)}{2}}+
q^{-\frac{n(3n+1)}{2}})\\ & = & 1-
q^{-1}-q^{-2}+q^{-5}+q^{-7}-q^{-12}-q^{-15}+ \cdots \end{eqnarray*}
\end{lemma}

Throughout this section quantities which can be easily re-expressed in
terms of the infinite product $\prod_{i=1}^{\infty} (1-\frac{1}{q^i})$
will often arise, and Lemma \ref{pent} gives arbitrarily accurate upper
and lower bounds on these products. Hence we will state bounds like
$\prod_{i=1}^{\infty} (1+\frac{1}{2^i}) = \prod_{i=1}^{\infty}
\frac{(1-\frac{1}{4^i})}{(1-\frac{1}{2^i})} \leq 2.4$ without explicitly
mentioning Euler's pentagonal number theorem on each occasion.

\subsection{$GL(n,q)$ and its relatives}
\label{GL}

To begin we discuss $GL(n,q)$. By a formula of Feit and Fine \cite{FF,M},
the number of conjugacy classes in $GL(n,q)$ is the coefficient of $t^n$
in the generating function \[ \prod_{i=1}^{\infty} \frac{1-t^i}{1-qt^i}.\]
Using clever reasoning and Euler's pentagonal number theorem, it is proved
in \cite{MR} that the number of conjugacy classes of $GL(n,q)$ is less
than $q^n$.

To this we add the following simple proposition.

\begin{prop} \label{asymgl}
\begin{enumerate}
\item For $q$ fixed, $lim_{n \rightarrow \infty} \frac{k(GL(n,q))}{q^n} = 1$.
\item $q^n-q^{n-1} \leq k(GL(n,q)) \leq q^n$. Thus $lim_{q \rightarrow \infty}
\frac{k(GL(n,q))}{q^n} = 1$, and the convergence is uniform in $n$.
\end{enumerate}
\end{prop}

\begin{proof} The generating function for conjugacy classes
of $GL(n,q)$ gives that $\frac{k(GL(n,q))}{q^n}$ is the coefficient of
$t^n$ in \[ \frac{1}{1-t} \prod_{i \geq 1}
\frac{1-t^i/q^i}{1-t^{i+1}/q^i}.\] For the first assertion, use Lemma
\ref{bign}. For the second assertion, the upper bound on $k(GL(n,q))$ was
mentioned earlier, and the lower bound holds since $GL(n,q)$ has
$q^n-q^{n-1}$ semisimple conjugacy classes, corresponding to the possible
characteristic polynomials. \end{proof}

{\it Remark:} In fact $\frac{k(GL(n,q))}{q^n}$ is even closer to 1 than
one might suspect from Proposition \ref{asymgl}. Indeed, \[ \frac{1}{1-t}
\prod_{i \geq 1} \frac{1-t^i/q^i}{1-t^{i+1}/q^i} - \frac{1}{1-t} \] is
analytic for all $|t|<q^{1/2}$ (subtracting the $(1-t)^{-1}$ removed the
pole at $t=1$). Thus Lemma \ref{maxmodulus} gives that for any $0<
\epsilon<1/2$, $\left| \frac{k(G)}{q^n}-1 \right| \leq
\frac{C_{q,\epsilon}}{q^{n(1/2-\epsilon)}}$ where $C_{q,\epsilon}$ is a
constant depending on $q$ and $\epsilon$ (which one could make explicit
with more effort). This is consistent with the fact (\cite{BFH},
\cite{MR}) that $k(GL(n,q))$ is a polynomial in $q$ with lead term $q^n$
and vanishing coefficients of $q^{n-1}, \cdots, q^{\lfloor \frac{n+1}{2}
\rfloor}$.

\vspace{2mm}

    Macdonald \cite{M} derived formulas for the number of conjugacy
classes of $SL(n,q), PGL(n,q)$ and $PSL(n,q)$ in terms of $k(GL(n,q))$. As
these will be used below it is useful to recall them. Let \[ \phi_r(n)=n^r
\prod_{p|n} (1-p^{-r}) \] where the product is over primes dividing $n$.
Thus $\phi_1(n)$ is Euler's $\phi$ function. Macdonald showed that
\[ k(SL(n,q)) = \frac{1}{q-1} \sum_{d|n,q-1} \phi_2(d) k(GL(n/d,q)), \]
\[ k(PGL(n,q)) = \frac{1}{q-1} \sum_{d|n,q-1}
 \phi_1(d) k(GL(n/d,q)), \]  and
\[ k(PSL(n,q)) = \frac{1}{(q-1) \gcd(n,q-1)} \sum_{d_1,d_2} \phi_1(d_1)
\phi_2(d_2) k(GL(\frac{n}{d_1d_2},q)) \] where the sum is over all pairs
of divisors $d_1,d_2$ of $q-1$ such that $d_1d_2$ divides $n$.

\begin{prop} \label{ksl} \begin{enumerate}
\item $q^{n-1} < k(SL(n,q)) \leq  2.5 q^{n-1}$.
\item $k(SL(n,q)) \leq q^{n-1}+3q^{n-2}$.
 Thus $lim_{q \rightarrow \infty} \frac{k(SL(n,q))}{q^{n-1}} = 1$, and
the convergence is uniform in $n$.
\item For $q$ fixed, $lim_{n \rightarrow \infty}
\frac{k(SL(n,q))}{q^{n-1}} = \frac{1}{1-1/q}$.
\end{enumerate}
\end{prop}

\begin{proof}   The lower bound holds by Theorem \ref{ss}.
From Macdonald's formula for $k(SL(n,q))$ one checks that $k(SL(2,q)) \le
q + 4$ and $k(SL(3,q)) \leq q^2+q+8$ (and that parts 1 and 2 hold for
$q=2,3$ and $n \leq 4$). So assume that $n \ge 4$.

If $q=2$, then the upper bound in part 1 follows by \cite[Lemma A.1]{MR}.
So assume also that $q \geq  3$. Since $k(GL(n/d,q)) \leq q^{n/d}$, it
follows that

\begin{eqnarray*} k(SL(n,q)) & \leq & \frac{1}{q-1} \sum_{d|n,q-1} d^2 q^{n/d}\\
& \leq & \frac{1}{q-1} [q^n + (q-1)^2 (q^{n/2}+q^{n/2-1}+\cdots)]\\ & \leq
& q^n/(q-1) + q^{n/2+1}.
\end{eqnarray*} For $q \geq 3, n \geq 4$, this is easily seen to be at
most $2.5 q^{n-1}$.

The upper bound in part 2 when $q=2$ follows by \cite[Lemma A.1]{MR}. For
$q \geq 3, n \geq 6$ one has that \[ q^n/(q-1) + q^{n/2+1} \leq q^{n-1} +
3 q^{n-2}, \] and one easily checks from Macdonald's formula that
$k(SL(n,q)) \leq q^{n-1}+3q^{n-2}$ for $q \geq 3, n \leq 5$.

For part 3, note from Macdonald's formula that \[
\frac{k(SL(n,q))}{q^{n-1}} = \frac{1}{1-1/q} \sum_{d|n,q-1}
 \phi_2(d) \frac{k(GL(n/d,q))}{q^n}. \] Since $q$ is fixed, it is clear
from Proposition \ref{asymgl} that only the $d=1$ term contributes in the
$n \rightarrow \infty$ limit, yielding the result.
\end{proof}

The following corollary concerns groups between $SL(n,q)$ and $GL(n,q)$ or
between $PSL(n,q)$ and $PGL(n,q)$.

\begin{cor} \label{PGLbetween}
\begin{enumerate}
\item Suppose that $SL(n,q) \subseteq H \subseteq GL(n,q)$, and let
$j$ denote the index of $H$ in $GL(n,q)$. Then \[ \frac{q^{n-1}(q-1)}{j}
\leq k(H) \leq \frac{q-1}{j} k(SL(n,q)) \leq \frac{q^n+3q^{n-1}}{j}. \]

\item Suppose that $PSL(n,q) \subseteq H \subseteq PGL(n,q)$,
and let $j$ denote the index of $H$ in $PGL(n,q)$. Then
\[ \frac{q^{n-1}}{j} \leq k(H) \leq \frac{\gcd(n,q-1)}{j} k(PSL(n,q))
\leq \frac{q^{n-1}+5q^{n-2}}{j}. \]
\end{enumerate}
\end{cor}

\begin{proof} Let $H$ be as in part 1 of the corollary. Then $k(H) \geq
\frac{k(GL(n,q))}{j} \geq \frac{q^{n-1}(q-1)}{j}$, where the first
inequality is Lemma \ref{Boblemma} and the second is the fact that
$GL(n,q)$ has $q^{n-1}(q-1)$ semisimple conjugacy classes. The inequality
$k(H) \leq \frac{q-1}{j} k(SL(n,q))$ comes from Lemma \ref{Boblemma}, and
Proposition \ref{ksl} yields the inequality $(q-1) k(SL(n,q)) \leq
q^n+3q^{n-1}$.

Let $H$ be as in part 2 of the corollary. Then $k(H) \geq
\frac{k(PGL(n,q))}{j} \geq \frac{q^{n-1}}{j}$, where the first inequality
is Lemma \ref{Boblemma} and the second is the fact that $PGL(n,q)$ has at
least $q^{n-1}$ conjugacy classes (clear from Macdonald's formula and the
fact that $GL(n,q)$ has at least $q^{n-1}(q-1)$ conjugacy classes). The
inequality $k(H) \leq \frac{\gcd(n,q-1)}{j} k(PSL(n,q))$ comes from Lemma
\ref{Boblemma}, and the inequality $\gcd(n,q-1) k(PSL(n,q)) \leq q^{n-1}+
5q^{n-2}$ follows from Macdonald's formula for $k(PSL(n,q))$ and an
analysis similar to that in Proposition \ref{ksl}.
\end{proof}

We close this section with the following exact formula for the number of
conjugacy classes of a group $H$ between $SL(n,q)$ and $GL(n,q)$. It
involves the quantity $\phi_2$ defined earlier in this section.

\begin{prop} \label{PGLbetween2} Suppose that $SL(n,q) \subseteq H \subseteq GL(n,q)$
and let $j$ denote the index of $H$ in $GL(n,q)$. Then \[ k(H) =
\frac{1}{j} \sum_{d|(j,n)} \phi_2(d) k(GL(\frac{n}{d},q)). \]
\end{prop}

\begin{proof} As in \cite{M}, to each conjugacy class of $GL(n,q)$,
there is associated a partition $\nu$ of $n$. To describe this recall that
conjugacy classes of $GL(n,q)$ are parametrized by associating to each
monic irreducible polynomial $p(x)$ over $F_q$ with non-zero constant term
a partition; if the partition corresponding to $p(x)$ has $m_i$ parts of
size $i$, then it contributes $deg(p) m_i$ parts of size $i$ to the
partition $\nu$. Throughout the proof we let $c_{\nu}$ denote the number
of conjugacy classes of $GL(n,q)$ of type $\nu$. We also let
$\nu_1,\cdots,\nu_r$ denote the parts of $\nu$.

    Given the partition $\nu$, we determine the number of conjugacy classes of
$GL(n,q)$ of type $\nu$ in $H$, multiply it by the number of $H$ classes
into which each such class splits (this number depends only on $\nu$) and
then sum over all $\nu$. Arguing as on pages 33-36 of \cite{M} shows that
the number of conjugacy classes of $GL(n,q)$ of type $\nu$ in $H$ is \[
\frac{\gcd(j,\nu_1,\cdots,\nu_r) c_{\nu}}{j}\] and that each such class
splits into $\gcd(j,\nu_1,\cdots,\nu_r)$ many $H$ classes. Thus the total
number of conjugacy classes of $H$ is \[ \frac{1}{j} \sum_{|\nu|=n}
\gcd(j,\nu_1,\cdots,\nu_r)^2 c_{\nu}. \] Arguing as on pages 36-37 of
\cite{M}, this can be rewritten as
\[ \frac{1}{j} \sum_{d|(j,n)} \phi_2(d) k(GL(\frac{n}{d},q)).\]
\end{proof}

\subsection{$GU(n,q)$ and its relatives} \label{U}

The paper \cite{MR} proves that \[ k(GU(n,q)) \leq q^n \prod_{i \geq 1}
\frac{1+1/q^i}{1-1/q^i} \leq 8.26 q^n.\] Proposition \ref{unlimit} gives
an asymptotic result.

\begin{prop} \label{unlimit}
\begin{enumerate}
\item For $q$ fixed, $lim_{n \rightarrow \infty} \frac{k(GU(n,q))}{q^n} =
\prod_{i \geq 1} \frac{1+1/q^i}{1-1/q^i}$.
\item $q^n+q^{n-1} \leq k(GU(n,q)) \leq q^n + Aq^{n-1}$ for a universal
constant $A$; one can take $A=16$ for $q=2$ and $A=7$ for $q \geq 3$. Thus
$lim_{q \rightarrow \infty} \frac{k(GU(n,q))}{q^n} = 1$, and the
convergence is uniform in $n$.
\end{enumerate}
\end{prop}

\begin{proof} Wall \cite{W} shows that $k(GU(n,q))$ is the coefficient
of $t^n$ in \[ \prod_{i=1}^{\infty} \frac{1+t^i}{1-qt^i}.\] Thus
$\frac{k(GU(n,q))}{q^n}$ is the coefficient of $t^n$ in \[ \frac{1}{1-t}
\prod_{i=1}^{\infty} \frac{1+t^i/q^i}{1-t^{i+1}/q^i}.\] For the first
assertion, use Lemma \ref{bign}.

For the second assertion, the lower bound comes from the easily proved
fact (essentially on page 35 of \cite{W}) that $GU(n,q)$ has $q^n+q^{n-1}$
many semisimple conjugacy classes. For the upper bound, the assertion when
$q=2$ is immediate from the fact that $k(GU(n,q)) \leq 8.26 q^n$. For $q
\geq 3$, recall that
\[ k(GU(n,q)) \leq q^n \prod_{i \geq 1} \frac{1+1/q^i}{1-1/q^i}.\]
Lemma \ref{pent} gives that \[ \prod_{i \geq 1} \frac{1+1/q^i}{1-1/q^i} =
\prod_{i \geq 1} \frac{1-1/q^{2i}}{(1-1/q^i)^2} \leq
\frac{1}{(1-1/q-1/q^2)^2} \leq 1+\frac{7}{q}, \] where the last inequality
is an easy calculus exercise.
\end{proof}

{\it Remarks:}
\begin{enumerate}
\item The value of the limit in part 1 of Proposition \ref{unlimit} is
8.25... when $q=2$.

\item As in the remark after Proposition \ref{asymgl}, the convergence of
$\frac{k(GU(n,q))}{q^n}$ to $\prod_{i \geq 1} \frac{1+1/q^i}{1-1/q^i}$ is
$O(q^{-n(1/2-\epsilon)})$ for any $0<\epsilon<1/2$. Indeed, subtracting
off the simple pole at $t=1$ from the generating function in Proposition
\ref{unlimit} gives that
\[ \frac{1}{1-t} \prod_{i=1}^{\infty} \frac{1+t^i/q^i}{1-t^{i+1}/q^i} -
\frac{1}{1-t} \prod_{i \geq 1} \frac{1+1/q^i}{1-1/q^i} \] is analytic for
all $|t|<q^{1/2}$, so the claim follows from Lemma \ref{maxmodulus}.
\end{enumerate}

Macdonald \cite{M} derived useful formulas for $k(SU(n,q))$, $k(PGU(n,q))$
and $k(PSU(n,q))$. These involve the quantity
\[ \phi_r(n) = n^r \prod_{p|n} (1-p^{-r}) \] where the product is over all
primes dividing $n$. He showed that
\[ k(SU(n,q)) = \frac{1}{q+1} \sum_{d|n,q+1} \phi_2(d) k(GU(n/d,q)),\]
\[ k(PGU(n,q)) = \frac{1}{q+1} \sum_{d|n,q+1} \phi_1(d) k(GU(n/d,q)),\]
and \[ k(PSU(n,q)) = \frac{1}{(q+1)\gcd(n,q+1)} \sum_{d_1,d_2} \phi_1(d_1)
\phi_2(d_2) k(GU(\frac{n}{d_1d_2},q)) \] where the sum is over all pairs
of divisors $d_1,d_2$ of $q+1$ such that $d_1d_2$ divides $n$.

\begin{prop} \label{SU}
\begin{enumerate}
\item $q^{n-1} \leq k(SU(n,q)) \leq 8.26 q^{n-1}$.

\item $k(SU(n,q)) \leq q^{n-1}+Aq^{n-2}$ for a universal constant $A$; one
can take $A=16$ for $q=2$ and $A=7$ for $q \geq 3$. Thus $lim_{q
\rightarrow \infty} \frac{k(SU(n,q))}{q^{n-1}} = 1$, and the convergence
is uniform in $n$.

\item For $q$ fixed, $lim_{n \rightarrow \infty} \frac{k(SU(n,q))}{q^{n-1}} =
\frac{1}{(1+1/q)} \prod_{i \geq 1} \frac{1+1/q^i}{1-1/q^i}$.
\end{enumerate}
\end{prop}

\begin{proof} The lower bound in part 1 is immediate from Theorem \ref{ss}.
The upper bounds in parts 1 and 2 will be proved together. For $n \leq 7$
the upper bounds are checked directly from Macdonald's formula for
$k(SU(n,q))$. Next suppose that $q=2$ and $n \geq 8$. Then Macdonald's
formula for $k(SU(n,q))$ and the upper bound on $k(GU(n,q))$ give that
\[ k(SU(n,2)) \leq \frac{8.26}{3} \left( 2^n + 8 \cdot 2^{n/3} \right) \leq 8.26
(2^{n-1}).\] For $q \geq 3$, we claim that $k(SU(n,q)) \leq (1+7/q)
q^{n-1}$, which also implies the upper bound in part 1. Suppose that $n
\geq 8$ is even (the case of odd $n$ is similar). Then Macdonald's formula
and part 2 of Proposition \ref{unlimit} give that $k(SU(n,q))$ is at most
\begin{eqnarray*} & & (1+7/q) \frac{\left( q^n +
(q+1)^2(q^{n/2} + \cdots +1) \right)}{q+1} \\
& \leq & (1+7/q) \left(q^{n-1}-q^{n-2}+q^{n-3} \cdots \pm 1 +
q^{n/2+1} + 2q^{n/2} + \cdots 2q+1 \right) \\
& \leq & (1+7/q) \cdot q^{n-1}. \end{eqnarray*}

Part 3 follows from part 1 of Proposition \ref{unlimit} and Macdonald's
formula for $k(SU(n,q))$ (argue as in the case of $SL$).
\end{proof}

Corollary \ref{ucor} gives bounds on $k(H)$ where $H$ is a group between
$SU(n,q)$ and $GU(n,q)$ or between $PSU(n,q)$ and $PGU(n,q)$.

\begin{cor} \label{ucor}
\begin{enumerate}
\item Suppose that $SU(n,q) \subseteq H \subseteq
GU(n,q)$ and that $j$ is the index of $H$ in $GU(n,q)$. \[
\frac{q^{n-1}(q+1)}{j} \leq k(H) \leq \frac{q+1}{j} k(SU(n,q)) \leq
\frac{q^n+Aq^{n-1}}{j}, \] where $A$ is a universal constant. One can take
$A=25$ for $q=2$ and $A=11$ for $q \geq 3$.

\item Suppose that $PSU(n,q) \subseteq H \subseteq
PGU(n,q)$ and that $j$ is the index of $H$ in $PGU(n,q)$.\[
\frac{q^{n-1}}{j} \leq k(H) \leq \frac{\gcd(n,q+1)}{j} k(PSU(n,q)) \leq
\frac{q^{n-1}+8q^{n-2}}{j}.\]
\end{enumerate}
\end{cor}

\begin{proof} Let $H$ be as in part 1 of the corollary. Then $k(H) \geq
\frac{k(GU(n,q))}{j} \geq \frac{q^{n-1}(q+1)}{j}$, where the first
inequality is Lemma \ref{Boblemma} and the second is the fact that
$GU(n,q)$ has $q^{n-1}(q+1)$ semisimple conjugacy classes. Lemma
\ref{Boblemma} gives that $k(H) \leq \frac{q+1}{j} k(SU(n,q))$. The
inequality $(q+1)k(SU(n,q)) \leq q^n+Aq^{n-1}$ with the stated $A$ values
follows from part 2 of Proposition \ref{SU}.

Let $H$ be as in part 2 of the corollary. Then $k(H) \geq
\frac{k(PGU(n,q))}{j} \geq \frac{q^{n-1}}{j}$, where the first inequality
is Lemma \ref{Boblemma} and the second is the fact that $PGU(n,q)$ has at
least $q^{n-1}$ conjugacy classes (clear from Macdonald's formula and the
fact that $GU(n,q)$ has $q^{n-1}(q+1)$ semisimple conjugacy classes). The
inequality $k(H) \leq \frac{\gcd(n,q+1)}{j} k(PSU(n,q))$ comes from Lemma
\ref{Boblemma}, and the inequality $\gcd(n,q+1) k(PSU(n,q)) \leq
q^{n-1}+8q^{n-2}$ follows from Macdonald's formula for $k(PSU(n,q))$ and
an analysis similar to that in Proposition \ref{SU}.
\end{proof}

\subsection{Symplectic groups} \label{Sp}

We next consider symplectic groups.   We treat the cases
$q$ odd and even  separately.

\begin{theorem} \label{Spodd} Let $q$ be odd.

\begin{enumerate}

\item $q^n \leq k(Sp(2n,q)) \leq q^n \prod_{i=1}^{\infty}
\frac{(1+\frac{1}{q^i})^4}{(1-\frac{1}{q^i})} \leq 10.8 q^n$.

\item $k(Sp(2n,q)) \leq q^n+Aq^{n-1}$ for a universal constant $A$; one can
take $A=30$ for $q=3$ and $A=12$ for $q \geq 5$. Thus \[ lim_{q
\rightarrow \infty} \frac{k(Sp(2n,q))}{q^n} = 1,\] and the convergence is
uniform in $n$.

\item For $q$ fixed, $lim_{n \rightarrow
\infty} \frac{k(Sp(2n,q))}{q^n} = \prod_{i=1}^{\infty}
\frac{(1+\frac{1}{q^i})^4}{(1-\frac{1}{q^i})}$.

\end{enumerate}

\end{theorem}

\begin{proof} The lower bound in part 1 is immediate from Theorem
\ref{ss}.

For $q$ odd, Wall \cite{W} shows that $k(Sp(2n,q))$ is the coefficient of
$t^n$ in the generating function \[ \prod_{i=1}^{\infty}
\frac{(1+t^i)^4}{1-qt^i}.\] Rewrite this generating function as \[
\prod_{i=1}^{\infty} \frac{1-t^i}{1-qt^i} \prod_{i=1}^{\infty}
\frac{(1+t^i)^4}{1-t^i}.\] Since all coefficients of powers of $t$ in the
second infinite product are non-negative, it follows that
\[ k(Sp(2n,q)) \leq \sum_{m=0}^n (Coef. \ t^{n-m} \ in \
\prod_{i=1}^{\infty} \frac{1-t^i}{1-qt^i}) (Coef. \ t^{m} \ in \
\prod_{i=1}^{\infty} \frac{(1+t^i)^4}{1-t^i}).\] Now
$\prod_{i=1}^{\infty} \frac{1-t^i}{1-qt^i}$ is the generating function
for the number of conjugacy classes in $GL(n,q)$. Hence the coefficient
of $t^{n-m}$ in it is at most $q^{n-m}$. It follows that \[ k(Sp(2n,q))
\leq q^n \sum_{m=0}^{n} \frac{1}{q^m} (Coef. \ t^{m} \ in \
\prod_{i=1}^{\infty} \frac{(1+t^i)^4}{1-t^i}).\] Since the coefficients
of $t^m$ in $\prod_{i=1}^{\infty} \frac{(1+t^i)^4}{1-t^i}$ are positive,
it follows that \begin{eqnarray*} k(Sp(2n,q)) & \leq & q^n
\sum_{m=0}^{\infty} \frac{1}{q^m} (Coef. \ t^{m} \ in \
\prod_{i=1}^{\infty} \frac{(1+t^i)^4}{1-t^i})\\ & = & q^n
\prod_{i=1}^{\infty} \frac{(1+\frac{1}{q^i})^4}{(1-\frac{1}{q^i})}.
\end{eqnarray*} The term $ \prod_{i=1}^{\infty}
\frac{(1+\frac{1}{q^i})^4}{(1-\frac{1}{q^i})}$ is visibly maximized among
odd prime powers $q$ when $q=3$. Then it becomes
$\frac{\prod_{i=1}^{\infty} (1-\frac{1}{9^{i}})^4}{\prod_{i=1}^{\infty}
(1-\frac{1}{3^i})^5} \leq 10.8 q^n$ by Lemma \ref{pent}.

The upper bound in part 2 follows from the upper bound in part 1, Lemma
\ref{pent}, and basic calculus (argue as in the unitary case).

For part 3, note that $\frac{k(Sp(2n,q))}{q^n}$ is the coefficient of
$t^n$ in
\[ \frac{1}{1-t} \prod_{i \geq 1} \frac{(1+t^i/q^i)^4}{(1-t^{i+1}/q^i)}.\]
Then use Lemma \ref{bign}.
\end{proof}

{\it Remark:} The value of the limit in part 3 of Theorem \ref{Spodd} is
10.7... when $q=3$.

\vspace{2mm}

Next we treat the symplectic group in even characteristic.

\begin{theorem} \label{Spev} Let $q$ be even.

\begin{enumerate}

\item \[ 1 + \sum_{n \geq 1} k(Sp(2n,q)) t^n = \prod_{i=1}^{\infty}
\frac{(1-t^{4i})}{(1-t^{4i-2}) (1-t^i) (1-qt^i)}.\]

\item $q^n \leq k(Sp(2n,q)) \leq q^n \prod_{i=1}^{\infty}
\frac{(1-1/q^{4i})}{(1-1/q^{4i-2}) (1-1/q^i)^2} \leq 15.2 q^n$.

\item $k(Sp(2n,q)) \leq q^n + Aq^{n-1}$ for a universal constant $A$; one can
take $A=29$ for $q=2$ and $A=5$ for $q \geq 4$. Thus
\[ lim_{q \rightarrow \infty} \frac{k(Sp(2n,q))}{q^n} = 1,\] and the
convergence is uniform in $n$.

\item For $q$ fixed, $lim_{n \rightarrow \infty} \frac{k(Sp(2n,q))}{q^n} =
\prod_{i=1}^{\infty} \frac{(1-1/q^{4i})}{(1-1/q^{4i-2}) (1-1/q^i)^2}$.
\end{enumerate}

\end{theorem}

\begin{proof} For the first assertion, one combines work of Wall \cite{W}
and Andrews' solution of the L-M-W conjecture \cite{A2} to obtain that \[
1 + \sum_{n \geq 1} k(Sp(2n,q)) t^n = \frac{\sum_{j=0}^{\infty}
t^{j(j+1)}}{\prod_{i=1}^{\infty} (1-t^i) (1-qt^{i})}.\] An identity of
Gauss (page 23 of \cite{A1}) states that \[ \sum_{n=0}^{\infty}
t^{n(n+1)/2} = \prod_{i=1}^{\infty} \frac{1-t^{2i}}{1-t^{2i-1}}, \] and
the first assertion follows.

For the second assertion, combining part 1 with the same trick as in the
odd characteristic case gives that
\begin{eqnarray*} k(Sp(2n,q)) & \leq & q^n \prod_{i=1}^{\infty}
\frac{(1-1/q^{4i})}{(1-1/q^2)^2(1-1/q^{4i-2})}\\
& \leq & q^n \prod_{i=1}^{\infty}
\frac{(1-1/2^{4i})}{(1-1/2^i)^2(1-1/2^{4i-2})}\\
& \leq & 15.2 q^n.
\end{eqnarray*} The last step used Lemma \ref{pent}. The lower bound in
the second assertion is immediate from Theorem \ref{ss}.

The proofs of parts 3 and 4 are analogous to the proofs of parts 2 and 3
in the odd characteristic case.
\end{proof}

{\it Remark:} The value of the limit in part 4 of Theorem \ref{Spev} is
15.1... when $q=2$.

\subsection{Orthogonal groups} \label{O} This section gives the results
for the orthogonal groups.   We assume that the dimension of
the underlying space is at least $3$ (almost all of the results
are valid for the two dimensional case as well, but the results
are trivial in that case and the lower bounds do not always hold
because the semisimple rank is $0$).

 First we treat the case of even dimension with $q$ odd.

\begin{theorem} \label{oddorthog} Let $q$ be odd.

\begin{enumerate}

\item $\frac{q^n}{2} \leq k(O^{\pm}(2n,q)) \leq 9.5 q^{n}$.

\item $k(O^{\pm}(2n,q)) \leq \frac{q^n}{2}+Aq^{n-1}$ for a universal constant
$A$; one can take $A=27$ for $q=3$ and $A=18$ for $q \geq 5$. Thus \[
lim_{q \rightarrow \infty} \frac{k(O^{\pm}(2n,q))} {q^n}= \frac{1}{2},\]
and the convergence is uniform in $n$.

\item For fixed $q$,
\begin{eqnarray*} & & lim_{n \rightarrow \infty}
\frac{k(O^{\pm}(2n,q))}{q^{n}}\\
& = & \frac{1}{4 \prod_{i=1}^{\infty} (1-1/q^{i})} \left[
\prod_{i=1}^{\infty} (1+1/q^{i-1/2})^4 +\prod_{i=1}^{\infty}
(1-1/q^{i-1/2})^4 \right] \end{eqnarray*}
\end{enumerate} \end{theorem}

\begin{proof} For the lower bound in part 1, Theorem \ref{ss} gives that
$SO^{\pm}(2n,q)$ has at least $q^n$ semisimple classes, and at most two of
these can fuse into one class in $O^{\pm}(2n,q)$. For the upper bound,
clearly $k(O^{\pm}(2n,q))$ is the sum/difference of
$\frac{k(O^+(2n,q))+k(O^-(2n,q))}{2}$ and
$\frac{k(O^+(2n,q))-k(O^-(2n,q))}{2}$. By upper bounding each of these
terms, we will upper bound $k(O^{\pm}(2n,q))$.

Wall \cite{W} shows that $k(O^+(2n,q))+k(O^-(2n,q))$ is the coefficient
of $t^{2n}$ in the generating function \[ \prod_{i=1}^{\infty}
\frac{(1+t^{2i-1})^4}{1-qt^{2i}}.\] Rewrite this generating function as
\[ \prod_{i=1}^{\infty} \frac{1-t^{2i}}{1-qt^{2i}} \prod_{i=1}^{\infty}
\frac{(1+t^{2i-1})^4}{1-t^{2i}} .\]

Arguing as in the proofs for the symplectic cases and using Lemma
\ref{pent}, the coefficient of $t^{2n}$ is at most
\begin{eqnarray*} & & q^n \sum_{m \geq 0} \frac{1}{q^m} \left( Coef. \ t^{2m} \
in \prod_{i \geq 1} \frac{(1+t^{2i-1})^4}{(1-t^{2i})} \right) \\
& = & \frac{q^n}{2} \sum_{m \geq 0} \frac{1}{q^m} \left( Coef. \ t^{2m} \
in \left[ \prod_{i \geq 1} \frac{(1+t^{2i-1})^4}{(1-t^{2i})} +
\prod_{i \geq 1} \frac{(1-t^{2i-1})^4}{(1-t^{2i})}   \right] \right)\\
& \leq & \frac{q^n}{2} \left[ \prod_{i \geq 1}
\frac{(1+t^{2i-1})^4}{(1-t^{2i})} +
\prod_{i \geq 1} \frac{(1-t^{2i-1})^4}{(1-t^{2i})} \right]_{|t=3^{-.5}}\\
& \leq & 16.3 q^n. \end{eqnarray*}

Wall \cite{W} shows that $k(O^+(2n,q))-k(O^-(2n,q))$ is the coefficient of
$t^{n}$ in \[ \prod_{i=1}^{\infty} \frac{(1-t^{2i-1})}{(1-qt^{2i})} .\]
Since this is analytic for $t<q^{-1} + \epsilon$, Lemmas \ref{maxmodulus}
and \ref{pent} imply an upper bound of \[ q^n \prod_{i=1}^{\infty}
\frac{(1+1/q^{2i-1})}{(1-1/q^{2i-1})} \leq q^n \prod_{i=1}^{\infty}
\frac{(1+1/3^{2i-1})}{(1-1/3^{2i-1})} \leq 2.4q^n.\] Combining this with
the previous paragraph gives that $k(O^{\pm}(2n,q)) \leq 9.5 q^n$.

For part 2, the $q=3$ case is immediate from part 1. For $q \geq 5$, the
upper bound on $\frac{k(O^+(2n,q))+k(O^-(2n,q))}{q^n}$ in the proof of
part 1 and the lower bound $k(O^{\pm}(2n,q)) \geq \frac{q^n}{2}$ yield
that $k(O^{\pm}(2n,q)$ is at most
\[ \frac{1}{2} \left[ \prod_{i
\geq 1} \frac{(1+1/q^{i-1/2})^4}{(1-1/q^i)} + \prod_{i \geq 1}
\frac{(1-1/q^{i-1/2})^4}{(1-1/q^i)} \right] - \frac{1}{2}.\] The result
follows from Lemma \ref{pent} (as in the unitary case) and basic calculus.

For the third assertion, $\frac{k(O^{+}(2n,q)) + k(O^-(2n,q))}{q^{n}}$ is
the coefficient of $t^{2n}$ in
\[ \frac{1}{1-t^2} \prod_{i=1}^{\infty}
\frac{(1+t^{2i-1}/q^{i-1/2})^4}{1-t^{2(i+1)}/q^i} .\] By Lemma \ref{bign},
as $n \rightarrow \infty$ this converges to \[ \frac{1}{2
\prod_{i=1}^{\infty} (1-1/q^{i})} \left[ \prod_{i=1}^{\infty}
(1+1/q^{i-1/2})^4 +\prod_{i=1}^{\infty} (1-1/q^{i-1/2})^4 \right]. \]
Recall that $k(O^{+}(2n,q)) - k(O^-(2n,q))$ is the coefficient of $t^{n}$
in \[ \prod_{i=1}^{\infty} \frac{(1-t^{2i-1})}{(1-qt^{2i})} .\] Since this
is analytic for $|t|<q^{-1/2}$, it follows from Lemma \ref{maxmodulus}
that
\[ lim_{n \rightarrow \infty} \frac{k(O^+(2n,q))
-k(O^-(2n,q))}{q^{n}}=0.\] \end{proof}

{\it Remark:} The value of the limit in part 3 of Theorem \ref{oddorthog}
is 8.14... when $q=3$.

\vspace{2mm}

To treat even dimensional special orthogonal groups in odd characteristic,
the following lemma will be helpful.

\begin{lemma} \label{uselem}  Let $q$ be odd and let $G=SO^{\pm}(n,q)$.
Let $C=g^G$.
Set $H=O^{\pm}(n,q)$ containing $G$.  The following are equivalent:
\begin{enumerate}
\item $C = g^H$;
\item $g$ leaves invariant an odd dimensional nondegenerate space $W$.
\item  Some Jordan block of $g$ corresponding to either the polynomial
$z+1$ or the polynomial $z-1$ has odd size.
\end{enumerate}
If all Jordan blocks of $g$ corresponding to both the polynomials $z \pm
1$ have even size, then $g^H$ is the union of two conjugacy classes of
$G$.
 \end{lemma}

\begin{proof}   Since $[H:G]=2$, the last statement follows from
the equivalence of the first three conditions.

Suppose that $C=g^H$.  It follows that $g$ centralizes some element $x \in
H \setminus{G}$.    Raising $x$ to an odd  power, we may assume that the
order of $x$ is a power of $2$ and in particular that $x$ is semisimple.
Since $\det x = -1$, it follows that the $-1$ eigenspace of $x$ is
nondegenerate and odd dimensional.  Thus (2) holds.

Conversely, assume (2).  Taking $x = -1$ and  $W$ on $1$ on $W^{\perp}$
shows that $C_H(g)$ is not contained in $G$, whence (1) holds.   Also, the
subspace of $W$ corresponding to either the $z -1$ or $z + 1$  space is
odd dimensional, whence some Jordan block has odd size.  Thus (2) implies
(3).

Finally assume (3).   By induction, we may assume that $g$
acts indecomposably (i.e. preserves no nontrivial orthogonal decomposition
 on the natural module).  If  $n$ is odd, then clearly (2) holds.
 So we may assume that $n$ is even. Replacing
$g$ by $-g$ (if necessary), we may assume that $g$ is unipotent.
 By \cite[Theorem 2.12]{LS},  it follows that $g$ either is a single
 Jordan block of odd size or has two Jordan blocks of even size.
Since (3) holds, the latter case cannot hold.   Thus, $g$ consists
of a single Jordan block of odd size,  whence (2) holds.
\end{proof}

\begin{theorem} \label{Soeodd} Let $q$ be odd.
\begin{enumerate}
\item $k(SO^+(2n,q)) + k(SO^-(2n,q))$ is the coefficient of $t^{2n}$ in
\[ \frac{3}{2} \prod_{i \geq 1} \frac{(1-t^{2i})^2}{(1-t^{4i})^2(1-qt^{2i})} +
\frac{1}{2} \prod_{i \geq 1} \frac{(1+t^{2i-1})^4}{(1-qt^{2i})} .\]
\item $k(SO^+(2n,q)) - k(SO^-(2n,q))$ is the coefficient of $t^{n}$ in
\[ 2 \prod_{i \geq 1} \frac{1}{(1+t^i)(1-qt^{2i})}.\]
\item $q^n \leq k(SO^{\pm}(2n,q)) \leq 7.5 q^n$.
\item  $k(SO^{\pm}(2n,q)) \leq  q^n + Aq^{n-1}$ for a universal
constant $A$; one can take $A=20$ for $q=3$ and $A=8$ for $q \geq 5$. Thus
\[ lim_{q \rightarrow \infty} \frac{k(SO^{\pm}(2n,q))}{q^n} = 1,\] and the
convergence is uniform in $n$.

\item For fixed $q$, \begin{eqnarray*} & & lim_{n \rightarrow \infty}
\frac{k(SO^{\pm}(2n,q))}{q^n}\\
& = & \left[ \frac{3}{4} \prod_{i \geq 1} \frac{(1-1/q^i)}{(1-1/q^{2i})^2}
+ \frac{1}{8} \prod_{i \geq 1} \frac{(1+1/q^{i-1/2})^4}{(1-1/q^i)} +
\frac{1}{8} \prod_{i \geq 1} \frac{(1-1/q^{i-1/2})^4}{(1-1/q^i)} \right].
\end{eqnarray*} \\
\end{enumerate}
\end{theorem}

\begin{proof} Clearly $k(SO^{+}(2n,q))+k(SO^-(2n,q))=2A+B$, where $A$ is
the sum over $O^{+}(2n,q)$ and $O^-(2n,q)$ of the number of classes which
have determinant 1 and split into two $SO$ classes, and $B$ is the sum
over $O^+(2n,q)$ and $O^-(2n,q)$ of the number of classes which have
determinant 1 and do not split into two $SO$ classes. Applying Lemma
\ref{uselem} and arguing as on pages 41-2 of \cite{W} gives that $A$ is
the coefficient of $t^{2n}$ in \[ \prod_{i \geq 1}
\frac{(1-t^{2i})^2}{(1-t^{4i})^2 (1-qt^{2i})}.\] (The factor of
$(1-t^{4i})^{-2}$ comes from the fact that the $z \pm 1$ partitions have
only even parts which must occur with even multiplicity, and the other
factor is precisely Wall's $F_0^+(t)$). To solve for $B$, note that $A+B$
is the sum of $O^+(2n,q)$ and $O^-(2n,q)$ of the number of classes which
have determinant 1. Such classes correspond to elements where the $z+1$
piece has even size, so arguing as on pages 41-2 of \cite{W} (using his
notation) gives that $A+B$ is the coefficient of $t^{2n}$ in
\begin{eqnarray*} & & \frac{1}{2} [ F^+_+(t) + F^+_+(-t)] F^+_-(t)
F_0^+(t) \\
& = & \frac{1}{2} \left[ \prod_{i \geq 1}
\frac{(1+t^{2i-1})^4}{(1-qt^{2i})} + \prod_{i \geq 1}
\frac{(1-t^{2i-1})^2(1+t^{2i-1})^2}{(1-qt^{2i})} \right].\end{eqnarray*}
Calculating $2A+B$ completes the proof of the first part of the theorem.

For the second assertion, apply Lemma \ref{uselem} and argue as on pages
41-2 of \cite{W} (using his notation) to conclude that the $O^+(2n,q)$
number - the $O^-(2n,q)$ number of conjugacy classes which have
determinant 1 and split is the coefficient of $t^{2n}$ in
\[ \prod_{i \geq 1} \frac{1}{(1-t^{4i})^2} F_0^-(t) =
\prod_{i \geq 1} \frac{1}{(1+t^{2i})(1-qt^{4i})}. \] Again applying Lemma
\ref{uselem} and arguing as on pages 41-2 of \cite{W}, one sees that the
$O^+(2n,q)$ number - the $O^-(2n,q)$ number of conjugacy classes which
have determinant 1 and do not split is $0$. The second assertion follows.

The lower bound in part 3 is immediate from Theorem \ref{ss}. For the
upper bound, it follows from part 1 and elementary manipulations that
$k(SO^{+}(2n,q))+k(SO^-(2n,q))$ is the coefficient of $t^{2n}$ in \[
\left[ \frac{3}{2} \prod_{i \geq 1} \frac{1-t^{4i-2}}{1-t^{4i}} +
\frac{1}{2} \prod_{i \geq 1} \frac{(1+t^{2i})(1+t^{2i-1})^4}{(1-t^{4i})}
\right] \prod_{i \geq 1} \frac{1-t^{2i}}{1-qt^{2i}}.\] It is not difficult
to see that the expression in square brackets in the previous equation has
all coefficients non-negative when expanded as a power series in $t$ (use
the fact that the coefficient of $t^{4i-2}$ in $(1+t^{2i-1})^4$ is 6).
Hence one can argue as in the Theorem \ref{oddorthog} to conclude that
$k(SO^{+}(2n,q))+k(SO^-(2n,q))$ is at most $q^n$ multiplied by \[ \left[
\frac{3}{2} \prod_{i \geq 1} \frac{1-t^{4i-2}}{1-t^{4i}} + \frac{1}{4}
\prod_{i \geq 1} \frac{(1+t^{2i})(1+t^{2i-1})^4}{(1-t^{4i})} + \frac{1}{4}
\prod_{i \geq 1} \frac{(1+t^{2i})(1-t^{2i-1})^4}{(1-t^{4i})} \right]
\] evaluated at $t=3^{-.5}$. This at most $9.3 q^n$.

By part 2 and the fact that $2 \prod_{i \geq 1}
\frac{1}{(1+t^i)(1-qt^{2i})}$ is analytic for $|t| < q^{-1} + \epsilon$,
it follows from Lemma \ref{maxmodulus} that
$k(SO^{+}(2n,q))-k(SO^-(2n,q))$ is at most
\[ 2q^n \prod_{i \geq 1} \frac{1}{(1-1/q^i)(1-1/q^{2i-1})} \leq
2q^n \prod_{i \geq 1} \frac{1}{(1-1/3^i)(1-1/3^{2i-1})} \leq 5.6 q^n.\]
This, together with the previous paragraph, completes the proof of the
third assertion.

For part 4, the $q=3$ case is immediate from part 1. For $q \geq 5$, the
proof of part 3 showed that $\frac{k(SO^{\pm}(2n,q))}{q^n}$ is at most
\begin{eqnarray*}
& & \frac{3}{2} \prod_{i \geq 1} \frac{1-1/q^{2i-1}}{1-1/q^{2i}} +
\frac{1}{4} \prod_{i \geq 1}
\frac{(1+1/q^{i})(1+1/q^{i-1/2})^4}{(1-1/q^{2i})}\\ & & + \frac{1}{4}
\prod_{i \geq 1} \frac{(1+1/q^{i})(1-1/q^{i-1/2})^4}{(1-1/q^{2i})} - 1.
\end{eqnarray*} Using Lemma \ref{pent} (as in the unitary case), the
result follows from basic calculus.

The proof of part 5 is nearly identical to the proof of part 3 in Theorem
\ref{oddorthog}.
\end{proof}

{\it Remark:} The value of the limit in part 5 of Theorem \ref{Soeodd} is
4.6... when $q=3$.

\vspace{2mm}

We next consider the odd dimensional case.

\begin{theorem} \label{Sooodd} Let $q$ be odd.

\begin{enumerate}

\item $q^n \leq k(SO(2n+1,q)) \leq q^n \prod_{i=1}^{\infty}
\frac{(1-1/q^{4i})^2}{(1-1/q^i)^3 (1-1/q^{4i-2})^2} \leq 7.1 q^n$.

\item $k(SO(2n+1,q)) \leq q^n + Aq^{n-1}$ for a universal
constant $A$; one can take $A=19$ for $q=3$ and $A=8$ for $q \geq 5$. Thus
\[ lim_{q \rightarrow \infty} \frac{k(SO(2n+1,q))} {q^n}= 1,\] and the
convergence is uniform in $n$.

\item For fixed $q$, $lim_{n \rightarrow
\infty} \frac{k(SO(2n+1,q))}{q^n} = \prod_{i=1}^{\infty}
\frac{(1-1/q^{4i})^2}{(1-1/q^i)^3 (1-1/q^{4i-2})^2}$.

\item $k(O(2n+1,q)) = 2 k(SO(2n+1,q))$. \end{enumerate} \end{theorem}

\begin{proof} Lusztig \cite{Lz} proves that $k(SO(2n+1,q))$ is the
coefficient of $t^n$ in the generating function \[ \left(
\sum_{j=0}^{\infty} t^{j(j+1)} \right)^2 \prod_{i=1}^{\infty}
\frac{1}{(1-t^i)^2 (1-qt^i)}.\] By a result of Gauss (page 23 of
\cite{A1}), this is equal to
\[ \prod_{i=1}^{\infty} \frac{(1-t^{4i})^2}{(1-t^{4i-2})^2 (1-t^i)^2
(1-qt^i)}.\] Using the same trick as in the unitary and symplectic cases
one sees that \[ k(SO(2n+1,q)) \leq q^n \prod_{i=1}^{\infty}
\frac{(1-1/q^{4i})^2}{(1-1/q^i)^3 (1-1/q^{4i-2})^2}. \]This is maximized
for $q=3$ for which Lemma \ref{pent} yields an upper bound of $7.1 q^n$.

The lower bound follows by Steinberg's result on the number of
semisimple classes -- see Theorem \ref{ss}.

The second part follows from part 1 (use Lemma \ref{pent} as in the
unitary case and basic calculus), and the third part is proved using the
same method used for the symplectic groups.

Since $O(2n+1,q) = \Z/2\Z  \times SO(2n+1,q)$, the fourth result is clear.
\end{proof}

{\it Remark:} The value of the limit in part 3 of Theorem \ref{Sooodd} is
7.0.. when $q=3$.

\vspace{2mm}

We now state similar results for the groups $\Omega^{\pm}(n,q)$.
The proofs of these results are somewhat long and are in \cite{FG5}.

Theorem \ref{omega1} treats the  even dimensional groups,
while Theorem \ref{omega2} treats the odd dimensional case.

It is convenient to define $\Omega^*(2n,q) = \Omega^+(2n,q)$
if $q \equiv 1 \mod 4$ or $n$ is even and $\Omega^-(2n,q)$
otherwise, and similarly for $SO$.

Note that for the even dimensional special orthogonal groups,
one will be a direct product of its center and $\Omega$, and so
the answer for $\Omega$ is precisely $1/2$ the answer for $SO$
(these are precisely the cases not mentioned in the next result).

\begin{theorem} \label{omega1}    Let $q$ be odd.
\begin{enumerate}
\item Set $j=2$ if $*=+$ and $j=1$ if $*=-$.
Then  $k(\Omega^*(2n,q))$ is the coefficient of $t^{2n}$ in
\begin{eqnarray*} \frac{3}{8} \prod_{i \geq 1}
\frac{1}{(1+t^{2i})^2(1-qt^{2i})} + \frac{1}{8} \prod_{i \geq 1}
\frac{(1+t^{2i-1})^4}{(1-qt^{2i})}\\ + \frac{3}{2} \frac{\prod_{i \ odd}
(1+t^i)^2}{\prod_{i \geq 1} (1-qt^{4i})} + j \frac{\prod_{i \ odd}
(1-t^{2i})} {\prod_{i \geq 1} (1-qt^{4i})}. \end{eqnarray*}

\item For $n \geq 2$, $\frac{q^n}{2} \leq k(\Omega^*(2n,q)) \leq 6.8 q^n$.

\item $k(\Omega^*(2n,q)) \leq \frac{q^n}{2} + A q^{n-1}$. One can
take $A=16$ for $q=3$ and $A=8.5$ for $q \geq 5$. Thus
\[ lim_{q \rightarrow \infty} \frac{k(\Omega^+(2n,q))}{q^n} =
\frac{1}{2},\] and the convergence is uniform in $n$.

\item For fixed $q$, \begin{eqnarray*} & & lim_{n \rightarrow \infty}
\frac{k(\Omega^*(2n,q))}{q^n} = \frac{1}{2} \cdot lim_{n \rightarrow
\infty} \frac{k(SO^*(2n,q))}{q^n} \\ & = & \frac{3}{8} \prod_{i \geq 1}
\frac{1}{(1+1/q^i)^2(1-1/q^i)} + \frac{1}{16} \prod_{i \geq 1}
\frac{(1+1/q^{i-1/2})^4}{(1-1/q^i)} \\ & & + \frac{1}{16} \prod_{i \geq 1}
\frac{(1-1/q^{i-1/2})^4}{(1-1/q^i)}.\end{eqnarray*}
\end{enumerate}
\end{theorem}

{\it Remark:} The value of the limits in part 4 of Theorem \ref{omega1}
 is 2.3... when $q=3$.

The following result is for odd dimensional groups.

\begin{theorem} \label{omega2} Suppose that $q$ is odd.
\begin{enumerate}
\item $k(\Omega(2n+1,q))$ is the coefficient of $t^{2n}$ in \[
\frac{3}{4t} \frac{\prod_{i \ odd} (1+t^i)^2}{\prod_{i \geq 1}
(1-qt^{4i})}  + \frac{1}{2} \prod_{i \geq 1}
\frac{(1-t^{8i})^2}{(1-t^{8i-4})^2 (1-t^{2i})^2 (1-qt^{2i})}.\]
\item For $n \geq 2$, $\frac{q^n}{2} \leq k(\Omega(2n+1,q)) \leq 7.3 q^n$.
\item $k(\Omega(2n+1,q)) \leq \frac{q^n}{2} + A q^{n-1}$. One can
take $A=11$ for $q=3$ and $A=5.5$ for $q \geq 5$. Thus
\[ lim_{q \rightarrow \infty} \frac{k(\Omega(2n+1,q))}{q^n} =
\frac{1}{2},\] and the convergence is uniform in $n$.
\item For fixed $q$, \begin{eqnarray*} & & lim_{n \rightarrow \infty}
\frac{k(\Omega(2n+1,q))}{q^n} = \frac{1}{2} \cdot lim_{n \rightarrow
\infty} \frac{k(SO(2n+1,q))}{q^n} \\ & = & \frac{1}{2} \prod_{i \geq 1}
\frac{(1-1/q^{4i})^2}{(1-1/q^{4i-2})^2 (1-1/q^i)^3}. \end{eqnarray*}
\end{enumerate}
\end{theorem}

In particular, we have:

\begin{cor}  Fix an odd prime power $q$.  Then
$$
\lim_{m \rightarrow \infty} \frac{k(\Omega^{\pm}(m,q))}{k(SO^{\pm}(m,q))} =
\frac{1}{2}.
$$
\end{cor}

We now turn  to orthogonal groups in  characteristic $2$.  Since
the odd dimensional orthogonal groups are isomorphic
to  symplectic groups, we need only consider the even dimensional
case.

\begin{theorem} \label{Oee} Let $q$ be even.

\begin{enumerate}

\item \[ 1 + \sum_{n \geq 1} t^n \left[ k(O^+(2n,q))+k(O^-(2n,q)) \right]
= \prod_{i=1}^{\infty} \frac{(1+t^i)(1+t^{2i-1})^2}{(1-qt^i)} .\]

\item \[ \frac{q^n}{2} \leq k(O^{\pm}(2n,q)) \leq 15q^n.\]

\item $k(O^{\pm}(2n,q)) \leq \frac{q^n}{2}+ A q^{n-1}$ for a universal
constant $A$; one can take $A=29$ for $q=2$ and $A=9$ for $q \geq 4$. Thus
\[ lim_{q \rightarrow \infty} \frac{k(O^{\pm}(2n,q))} {q^n}=
\frac{1}{2},\] and the convergence is uniform in $n$.

\item For fixed $q$,
$lim_{n \rightarrow \infty} \frac{k(O^{\pm}(2n,q))}{q^n} = \frac{1}{2}
\prod_{i \geq 1} \frac{(1+1/q^i)(1+1/q^{2i-1})^2}{(1-1/q^i)}$.
\end{enumerate} \end{theorem}

\begin{proof} Combining \cite{W} and \cite{A2} shows that
$k(O^+(2n,q))+k(O^-(2n,q))$ is the coefficient of $t^n$ in the generating
function \[ \frac{\sum_{j=-\infty}^{\infty} t^{j^2}}{\prod_{i=1}^{\infty}
(1-t^i) (1-qt^i)}.\] The first assertion now follows from the following
special case of Jacobi's triple product identity (page 21 of \cite{A1}):
\[ \sum_{n=-\infty}^{\infty} t^{n^2} = \prod_{i=1}^{\infty}
(1-t^{2i})(1+t^{2i-1})^2 \]

Note that when the numerator of the generating function of part 1 is
expanded as a series in $t$, all coefficients are positive. Arguing as for
the unitary and symplectic groups gives that

\begin{eqnarray*} k(O^+(2n,q))+k(O^-(2n,q)) & \leq & q^n
\prod_{i \geq 1} \frac{(1+1/q^i)(1+1/q^{2i-1})^2}{(1-1/q^i)} \\
& \leq & q^n
\prod_{i \geq 1} \frac{(1+1/2^i)(1+1/2^{2i-1})^2}{(1-1/2^i)} \\
& \leq & 25.6 q^n. \end{eqnarray*}

From \cite{W}, $k(O^+(2n,q))- k(O^-(2n,q))$ is the coefficient of $t^n$ in
\[ \prod_{i=1}^{\infty} \frac{1-t^{2i-1}}{1-qt^{2i}}.\] Since this is
analytic for $|t|<q^{-1}+\epsilon$, Lemma \ref{maxmodulus} gives that
$k(O^+(2n,q))- k(O^-(2n,q))$ is at most
\[ q^n \prod_{i=1}^{\infty} \frac{(1+1/q^{2i-1})}{(1-1/q^{2i-1})} \leq q^n
\prod_{i=1}^{\infty} \frac{(1+1/2^{2i-1})}{(1-1/2^{2i-1})} \leq 4.2q^n .\]
Combining this with the the previous paragraph yields the upper bound in
part 2.

For the lower bound in part 2, $SO^{\pm}(2n,q)$ is simply connected. Thus
by Steinberg's theorem, the number of semisimple classes is exactly $q^n$
and so there are at least $q^n/2$ in $O^{\pm}(2n,q)$ (because the index
is 2, at most two classes fuse into one).

The third and fourth parts are proved by the same method as in Theorem
\ref{oddorthog}.
\end{proof}

{\it Remark:} The value of the limit in part 4 of Theorem \ref{Oee} is
12.7.. when $q=2$.

\vspace{2mm}

Finally, we treat even characteristic special orthogonal groups.

\begin{theorem} \label{SOeven} Let $q$ be even.
\begin{enumerate}
\item \begin{eqnarray*} & & 2+\sum_{n \geq 1} t^n \left[ k(SO^+(2n,q)) +
k(SO^-(2n,q)) \right]\\ & = & \left[ \frac{1}{2} \prod_{i \ odd}
\frac{(1+t^i)^2}{(1-t^i)} + \frac{3}{2} \prod_{i=1}^{\infty}
\frac{1}{(1+t^i)} \right] \prod_{i=1}^{\infty} \frac{1}{(1-qt^i)}.
\end{eqnarray*}
\item \[ 2+\sum_{n \geq 1} t^n \left[ k(SO^+(2n,q)) - k(SO^-(2n,q)) \right] =
2 \prod_{i=1}^{\infty} \frac{1}{(1+t^i)(1-qt^{2i})}.\]
\item \[ q^n \leq k(SO^{\pm}(2n,q)) \leq 14 q^n.\]
\item $k(SO^{\pm}(2n,q)) \leq q^n + Aq^{n-1}$ for a universal constant $A$;
one can take $A=26$ for $q=2$ and $A=5$ for $q \geq 4$. Thus \[ lim_{q
\rightarrow \infty} \frac{k(SO^{\pm}(2n,q))}{q^n}=1,\] and the convergence
is uniform in $n$.
\item For fixed $q$, $lim_{n \rightarrow \infty}
\frac{k(SO^{\pm}(2n,q))}{q^n}$ is equal to  \[ \frac{1}{4} \prod_{i \ odd}
\frac{(1+1/q^i)^2}{(1-1/q^i)} \prod_{i=1}^{\infty} \frac{1}{(1-1/q^i)} +
\frac{3}{4} \prod_{i=1}^{\infty} \frac{1}{(1-1/q^{2i})}.\]
\end{enumerate}
\end{theorem}

\begin{proof} For part 1, it follows from \cite{A2} and
\cite{Lz} that if $k_1(SO^{\pm}(2n,q))$ is the number of unipotent
conjugacy classes of $SO^{\pm}(2n,q)$, then \begin{eqnarray*} & & 1 +
\sum_{n \geq 1} t^n \left[ k_1(SO^+(2n,q)) + k_1(SO^-(2n,q)) \right]\\
& = & \frac{ \frac{1}{2} + \sum_{n \geq 1} t^{n^2}}{\prod_{i \geq 1}
(1-t^i)^2} + \frac{3}{2} \prod_{i \geq 1} \frac{1}{(1-t^{2i})} - 1.
\end{eqnarray*}

We claim that a conjugacy class of $O^{\pm}(2n,q)$ with empty $z-1$ piece
splits in $SO^{\pm}(2n,q)$, and that a conjugacy class of $O^{\pm}(2n,q)$
with non-empty $z-1$ piece splits in $SO^{\pm}(2n,q)$ if and only if a
unipotent element with that $z-1$ piece splits in the $SO$ (possibly of
lower dimension) which contains it.    Let $x \in O^{\pm}(2n,q)$. Write
$V=V_1 \perp V_2$ where $V_1$ is the kernel of $(x-1)^{2n}$. Let $x_i$
denote the element of $O(V_i)$ that is the restriction  of $x$ to $V_i$.
Thus, the centralizer of $x$ is the direct product of the centralizers of
$x_i$ in $O(V_i)$. Working  over the algebraic closure  we see that the
centralizer of $x_2$ in $O(V_2)$ is isomorphic to the centralizer of some
element of $GL(d)$ where $2d = \dim V_2$.   In particular,  the
centralizer of  $x_2$ is connected and  so is contained in $SO(V_2)$.
Thus,  if $V_1=0$, the class of $x$ splits.  If $V_1 \ne 0$, then the
class of  $x$ splits if and only  if the class of $x_1$ splits in
$O(V_1)$.   This proves the claim.

  Thus \begin{eqnarray*} & & 2 + \sum_{n \geq 1} t^n \left[
k(SO^+(2n,q)) + k(SO^-(2n,q)) \right] \\ & = &\left( 2 + \sum_{n \geq 1}
t^n \left[ k_1(SO^+(2n,q)) + k_1(SO^-(2n,q)) \right] \right)
\prod_{i=1}^{\infty} \frac{(1-t^i)}{(1-qt^i)}, \end{eqnarray*} where the
term $\prod_{i \geq 1} \frac{(1-t^i)}{(1-qt^i)}$ is the even
characteristic analog of $F_0^+(t)$ from page 41 of Wall \cite{W} (and is
derived the same way). Plugging in the generating function for
$k_1^{\pm}(SO(2n,q))$ and using the an identity of Gauss (page 21 of
\cite{A1}) that \[ \frac{1}{2} + \sum_{n \geq 1} t^{n^2} = \frac{1}{2}
\prod_{i \ even} (1-t^i) \prod_{i \ odd} (1+t^i)^2,\] part 1 follows by
elementary simplifications.

For part 2, arguing as in part 1 gives that \begin{eqnarray*} & & 2 +
\sum_{n \geq 1} t^n \left[ k(SO^+(2n,q)) - k(SO^-(2n,q)) \right]\\ & = &
\left( 2 + \sum_{n \geq 1} t^n \left[ k_1(SO^+(2n,q)) - k_1(SO^-(2n,q))
\right] \right) \prod_{i=1}^{\infty}
\frac{(1-t^i)}{(1-qt^{2i})},\end{eqnarray*} where the term $\prod_{i \geq
1} \frac{(1-t^i)}{(1-qt^{2i})}$ is the even characteristic analog of
$F_0^-(t)$ from page 42 of Wall \cite{W} (and is derived the same way).
Page 153 of \cite{Lz} gives that
\[ 2 + \sum_{n \geq 1} t^n \left[ k_1(SO^+(2n,q)) - k_1(SO^-(2n,q)) \right] =
2 \prod_{i=1}^{\infty} \frac{1}{(1-t^{2i})},\] so part 2 follows.

The lower bound in part 3 is immediate from Theorem \ref{ss}.
For the upper bound,
first note from part 1 that $k(SO^+(2n,q)) + k(SO^-(2n,q))$ is the
coefficient of $t^n$ in \[ \left[ \frac{1}{2} \prod_{i \ odd}
\frac{(1+t^i)^2}{(1-t^i)} \prod_{i \geq 1} \frac{1}{(1-t^i)} + \frac{3}{2}
\prod_{i \geq 1} \frac{1}{(1-t^{2i})} \right] \prod_{i=1}^{\infty}
\frac{(1-t^i)}{(1-qt^i)}.
\] Arguing as in Theorem \ref{Spodd}, shows that $k(SO^+(2n,q)) + k(SO^-(2n,q))$
is at most $q^n$ multiplied by \[ \frac{1}{2} \prod_{i \ odd}
\frac{(1+t^i)^2}{(1-t^i)} \prod_{i \geq 1} \frac{1}{(1-t^i)} + \frac{3}{2}
\prod_{i \geq 1} \frac{1}{(1-t^{2i})}
\] evaluated at $t=1/q$. This is maximized at $q=2$, and is at most
$15$. Using the fact that $k(SO^{\pm}(2n,q)) \geq q^n$, it follows that
$k(SO^{\pm}(2n,q)) \leq 14q^n$.

By part 2 and the fact that $2 \prod_{i \geq 1}
\frac{1}{(1+t^i)(1-qt^{2i})}$ is analytic for $|t|<q^{-1}+\epsilon$, it
follows from Lemma \ref{maxmodulus} that $k(SO^+(2n,q)) - k(SO^-(2n,q))$
is at most \[ 2q^n \prod_{i \geq 1} \frac{1}{(1-1/q^i)(1-1/q^{2i-1})} \leq
2q^n \prod_{i \geq 1} \frac{1}{(1-1/2^i)(1-1/2^{2i-1})} \leq 17 q^n.\]
This, together with the previous paragraph, completes the proof of the
third assertion.

For part 4, the proof of part 3 yields that
$\frac{k(SO^{\pm}(2n,q))}{q^n}$ is at most \[ \frac{1}{2} \prod_{i \ odd}
\frac{(1+1/q^i)^2}{(1-1/q^i)} \prod_{i \geq 1} \frac{1}{(1-1/q^i)}+
\frac{3}{2} \prod_{i \geq 1} \frac{1}{(1-1/q^{2i})} - 1.
\] Using Lemma \ref{pent} (as in the unitary case), this upper bound is at
most $2+\frac{A}{q}$ for a universal constant $A$. Since
$k(SO^{\pm}(2n,q)) \geq q^n$ by part 3, the result follows.

The proof of part 5 is nearly identical to the proof of part 3 in Theorem
\ref{oddorthog}. \end{proof}

{\it Remark}: The limit in part 5 of Theorem \ref{SOeven} is 7.4.. when
$q=2$.

\section{Tables of Conjugacy Class Bounds} \label{exceptional}

We tabulate some of the results in the previous section and summarize the
corresponding results for exceptional groups. There are exact formulas for
these class numbers and we refer the reader to \cite{Lu}. See also \S8.18
of \cite{H} and the references therein.

Here are the results for the exceptional groups.  We give a polynomial
upper bound for each type of exceptional group (this upper bound is valid
for both the adjoint and simply connected forms of the group). \\

\centerline
{{\sc Table 1} \quad Class Numbers for Exceptional Groups}
\begin{center}
\begin{tabular}{|c||c|c|} \hline
$G$ & $k(G) \le $ & Comments   \\ \skipa \hline \hline
${^2}B_2(q)$ & $q + 3$ & $q=2^{2m + 1}$ \\ \hline
${^2}G_2(q)$ & $q + 8$ & $q=3^{2m+1}$ \\ \hline
$G_2(q)$  &   $q^2 + 2q + 9$  & \\ \hline
${^2}F_4(q)$ & $q^2 + 4q + 17$ &  $q=2^{2m + 1}$ \\ \hline
${^3}D_4(q)$ & $q^4+q^3+q^2+q+6$ &   \\ \hline
$F_4(q) $ &  $q^4 + 2q^3 + 7q^2 + 15q + 31$ & \\ \hline
$E_6(q)$ & $ q^6 + q^5 + 2q^4 + 2q^3 + 15q^2 + 21q + 60$ & \\ \hline
${^2}E_6(q)$ & $    q^6 + q^5 + 2q^4 + 4q^3 + 18q^2 + 26q + 62$  &  \\ \hline
$E_7(q) $ &  $q^7 + q^6 + 2q^5 + 7q^4 + 17q^3 + 35q^2 + 71q + 103$ & \\ \hline
$E_8(q)$ &  $ q^8 + q^7 + 2q^6 + 3q^5 + 10q^4 + 16q^3 + 40q^2 + 67q + 112 $ & \\ \hline
\end{tabular}
\end{center}

\bigskip

In particular, we see that $1 \le k(G)/q^r \rightarrow 1$ as $q
\rightarrow \infty$ in all cases.   Also,  $q^r  < k(G) \le q^r  +
14q^{r-1}$ and $k(G) \le 8q^r$ in all cases.

Below we summarize some of the results of the previous section on various
(but not all) forms of the classical groups.  These results all follow
from the previous section.   \\

\centerline
{{\sc Table 2} \quad  Class Numbers for Classical Groups}
\vspace{0.3cm}
\begin{center}
\begin{tabular}{|c||c|c|} \hline
 $G$ & $k(G)  \le $ &   Comments   \\ \skipa \hline \hline
$SL(n,q)$  & $ 2.5 q^{n-1}$ &  \\ \hline $SU(n,q)$  & $8.26 q^{n-1}$  &
\\  \hline $Sp(2n,q)$ &  $10.8q^n$ &   $q$ odd   \\ \hline
$Sp(2n,q)$  &  $15.2q^n$ &   $q$ even   \\ \hline
 $SO(2n+1,q)$ & $7.1q^n$    &   $q$ odd \\ \hline
 $\Omega(2n+1,q)$ & $7.3q^n$ &   $q$ odd \\ \hline
 $SO^{\pm}(2n,q)$ &  $7.5q^n$ &   $q$ odd \\ \hline
$\Omega^{\pm}(2n,q)$ &  $6.8q^n$         &   $ q$ odd \\ \hline
$O^{\pm}(2n,q)$ &  $9.5q^n$         &   $q$ odd \\ \hline
$SO^{\pm}(2n,q)$ &  $14q^n$         &   $q$ even \\ \hline
$O^{\pm}(2n,q)$ &  $15q^n$         &   $q$ even \\ \hline
\end{tabular}
\end{center}

\bigskip
In the next table,  we give bounds of the form $Aq^r + Bq^{r-1}$.  We
exclude $q=2$ or  $3$ in some cases -- a similar bound follows from the
previous table in those cases. \\

\centerline
{{\sc Table 3} \quad  Class Numbers for Classical Groups  II}~
\vspace{0.3cm}
\begin{center}
\begin{tabular}{|c||c|c|} \hline
 $G$ & $k(G)  \le $ &   Comments   \\ \skipa \hline \hline
$SL(n,q)$  & $ q^{n-1} + 3q^{n-2}$ &  \\ \hline
 $PGL(n,q)$ & $q^{n-1} +
5q^{n-2}$ &  \\  \hline $SU(n,q)$  & $q^{n-1} + 7q^{n-2}$  &  $q > 2$
\\  \hline
$PGU(n,q)$ & $q^{n-1} + 8q^{n-2}$   &  $q > 2$  \\  \hline
$Sp(2n,q)$ &  $q^{n} + 12q^{n-1}$ &   $q > 3$ odd   \\ \hline
$Sp(2n,q)$ & $q^{n} + 5q^{n-1}$ &   $q>2$ even   \\ \hline
$SO(2n+1,q)$ & $q^{n} + 8q^{n-1}$    &   $q > 3$ odd \\ \hline
$\Omega (2n+1,q)$ & $(1/2)q^{n} + 5.5q^{n-1}$    &   $q > 3$ odd
\\ \hline $SO^{\pm}(2n,q)$ &  $q^{n} + 8q^{n-1}$         &   $q > 3$ odd \\ \hline
 $\Omega^{\pm}(2n,q)$ &  $(1/2)q^n  + 8.5q^{n-1}$         &   $ q > 3$ odd \\ \hline
$O^{\pm}(2n,q)$ &  $(1/2)q^n + 18 q^{n-1}$           &   $q > 3$ odd \\ \hline
$SO^{\pm}(2n,q)$ &  $q^{n} + 5q^{n-1}$         &   $q > 2$ even \\ \hline
$O^{\pm}(2n,q)$ &  $(1/2)q^n +  9 q^{n-1}$         &   $q > 2$ even\\ \hline
\end{tabular}
\end{center}

\section{Conjugacy Classes in Almost Simple Groups} \label{class results}

We use the results of the previous sections to obtains bounds
on the class numbers for Chevalley groups.  These bounds are
close to best possible, but we improve these a bit in \cite{FG5}.
These bounds are used in \cite{LOST} to finish the proof of the
Ore conjecture and are more than sufficient for our proof of
the Boston-Shalev conjecture on derangements.

We first prove Theorem \ref{A}.

\begin{proof}   Since the number of semisimple classes in $G^F$
is at least $q^r$, (2) implies (3).  Since any $F$-stable class
intersects $G^F$, the number of $F$-stable classes is at most
$k(G^F)$ and so (1) implies (4).  Thus, it suffices to prove
(1) and (2).

First assume that $G$ has type A. Apply  Theorems \ref{ksl} and \ref{SU}
and Corollaries \ref{PGLbetween} and \ref{ucor} to conclude that (1) and
(2) hold in this case.

If $G$ is exceptional, then the results follow by the Table
in the previous section.

Now assume that $G$ has type B, C or D.  Since the center has order
at most $4$ in all cases, to prove (1), it suffices to consider
any form of the group (this may alter the constant but only by a bounded
amount, and in fact it changes only very little).  Moreover,
in characteristic $2$, the adjoint and simply connected groups
are the same and so there is nothing to prove.  So we may
assume that the characteristic of the field is odd.

Suppose that $G$ has type B.  The results have been
proved for the group of adjoint type with constants $7.1$
and $19$ respectively.  Suppose that $G$ is simply connected.
Then $k(G) \le 2 k(\Omega)$, whence the results hold by
Theorem \ref{omega2}.

Next suppose that $G$ has type C.  We have already proved the result for
the simply connected group.  So assume that $G$ is the adjoint form. Let
$H = Sp(2r,q)$.   Then $k(H/Z(H)) < k(H) \le 10.8q^r$, whence
$k(G) \le 21.6q^r$, and so (1) holds.

The number of semisimple classes in $H/Z(H)$ is at most
$(q^r + t)/2$, where $t$ is the number of $H$-semisimple classes invariant
under multiplication by $Z(H)$.   These correspond to monic polynomials of
degree $r$ in $x^2$ with the set of roots invariants under inversion. The
number of such is  $q^{(r-1)/2}$ if $r$ is odd and $q^{r/2}$ if $r$ is
even.  Thus, $k(H/Z(H)) \le (1/2)q^r  +  31q^{r-1} $.  Since $G/H$
has order $2$, this implies that $k(G) \le q^r + 62 q^{r-1}$, and so
(2) holds.

Finally, consider the case that $G$ has type D (and we may
assume that $r \ge 4$).    Let
$H= P\Omega^{\pm}(2r,q)$ be the simple group corresponding to
$G$.  Using
Theorem \ref{omega1}, we see that $k(P\Omega^{\pm}(2r,q)) \le 6.8 q^r$,
whence a straightforward argument shows that $k(G) \le 4 k(H)  \le 27.2 q^r$.

Arguing as in the case of type C, we see that
$k(P\Omega^{\pm}(2r,q)) \le (1/4)q^r  +16q^{r-1} + q^{r/2}$.
Thus, $k(G) \le 4k(P\Omega^{\pm}(2r,q)) \le q^r + 68q^{r-1}$.
 \end{proof}

 Note that for simply connected groups or groups of adjoint
 type, we can do a bit better (and in particular for the simple
 groups).   This follows by the proof  above.

 \begin{cor} \label{simple1}  Let $G$ be a simply  connected
simple algebraic group of rank $r$ over a field of positive
characteristic. Let $F$ be a Steinberg-Lang endomorphism of $G$ with $G^F$
a finite Chevalley group over the field  $\F_q$.
   Then
 \begin{enumerate}
 \item   $k(G^F)  \le   15.2q^r$.
 \item   $k(G^F ) \le q^r + 40q^{r-1}$.
 \end{enumerate}
 \end{cor}

 Similarly, the proof of Theorem \ref{A} also shows:

\begin{cor} \label{simple2} Let $G$ be a
simple algebraic group of adjoint type  of rank $r$ over a field of positive
characteristic. Let $F$ be a Steinberg-Lang endomorphism of $G$ with $G^F$
a finite Chevalley group over the field  $\F_q$.   Let $S$ be the socle
of $G^F$ and assume that $S \le H \le G$.
   Then
 \begin{enumerate}
 \item   $k(H)  \le   27.2q^r$.
 \item   $k(H ) \le q^r + 68q^{r-1}$.
 \item  The number of non-semisimple classes of $H$
 is at most $68q^{r-1}$.
 \end{enumerate}
 \end{cor}

 Corollary  \ref{simple2} also leads to the following result which was used
 in \cite{GR}.

 \begin{prop}  Let $G$ be a finite almost simple group. Then $k(G) \le |G|^{.41}$.
 \end{prop}

\begin{proof} For all cases other than the alternating and symmetric
groups, this follows from Theorem \ref{simple1} together with bounds on
the size of the outer automorphism groups and a computer computation for
small cases. For the symmetric groups, the result follows without
difficulty from the two bounds $k(S_n) \leq \frac{\pi}{\sqrt{6(n-1)}}
e^{\pi \sqrt{\frac{2n}{3}}}$ (\cite{VW}, p. 140) and $n! \geq (2
\pi)^{1/2} n^{n+1/2} e^{-n+1/(12n+1)}$ (\cite{Fe}, p.52). By Corollary
\ref{altsym}, $k(A_n) \leq k(S_n)$, and the two bounds also imply that
$k(S_n) \leq \left( \frac{n!}{2} \right)^{.41}$ for $n \geq 6$. For $n=5$,
$k(A_5)=5 < (60)^{.41}$ and $k(S_5)=7 < (120)^{.41}$.
 \end{proof}

Typically, the $.41$ can be replaced by a much smaller number but the
example of $G=S_5$ shows that one cannot do much better.

We now consider general almost simple Chevalley groups $G$. Now we have to
deal with all types of outer automorphisms. The proof we give below shows
that in fact for large $q$, $k(G)$ is close to $q^r/e$ where
$e=[\Inndiag(S):G \cap \Inndiag(S)]$.

We first need a lemma.   See \cite{GLS3} for basic results
about automorphisms of Chevalley groups.

\begin{lemma}  Let $S$ be a simple Chevalley group over the field
of $q$ elements of rank $r \ge 2$.  Let $x \in \aut(S)$ with $x$ not an
inner diagonal automorphism of $S$.    Let $S \le H \le \Inndiag(S)$.
 Then the number of $x$
stable classes in $H$ is at most $Dq^{r-1}$ for some universal constant $D$.
\end{lemma}

\begin{proof}
We may assume that $x$ has prime order $p$
modulo the group of inner diagonal automorphisms.

By Corollary \ref{simple2},  it suffices to consider semisimple classes
of $S$ (since the number of  nonsemisimple classes is at most $Aq^{r-1}$
for a universal constant $A$).
  If $x$ is in the coset of a Lang-Steinberg
automorphism, then Shintani descent gives a much better bound.
In any case, the stable classes will be in bijection with those
in the centralizer and so there are at most $Cq^{r/2}$ invariant
classes, whence the bound holds in this case.

The remaining cases are where $x$ induces a graph automorphism. We lift to
the central cover $T$ of $S$ and let $H_0$ be the lift of $H$. It suffices
to prove the result for $H_0$.  By considering irreducible representations
of $T$, we see that  the number of stable semisimple classes in $T$ is
$q^{r'}$ where $r-1 \ge r'$ is the number of orbits of $x$ on the Dynkin
diagram of $S$.    Similarly, for each $x$-invariant coset of $H_0/T$,
there are are most $q^{r'}$ invariant classes in each of those cosets.  If
$S$ is not of type A, there are at most $4$ cosets.  If $S$ is of type A,
there are at most $2$ invariant cosets.    Thus, there are most $4q^{r-1}$
invariant semisimple classes.
\end{proof}

We can now prove:

\begin{theorem}  Let $G$ be almost simple with socle
$S$ that is a Chevalley group defined over the field of
$q$ elements and has rank $r$.   There is an absolute constant
$D$ such that $k(G) \le  q^r + D (\log q) q^{r-1} \le D'q^r$.
\end{theorem}

\begin{proof}   Let $H$ be the subgroup of $G$ consisting
of inner diagonal automorphisms.     Let $X$ be the full group of inner
diagonal automorphisms of $S$. Then $X = Y^F$ where $Y$ is the
corresponding simple algebraic group of adjoint type, and so by Corollary \ref{same},
the number of
semisimple classes in $H$ is precisely $[X:H]^{-1}$ times the number of
semisimple classes in $X$.   The number of non-semisimple classes in $H$
is at most $[X:H]$ times the number of non-semisimple classes of $X$.
 If $G$ is not of type A, then $[X:H] \le 4$ and so by Theorem \ref{A},
$k(H) \le q^r + Eq^{r-1}$ for some universal constant $E$.
If $G$ is of type A, we apply Corollaries \ref{PGLbetween} and
\ref{ucor}
to conclude this as well.

  First consider the case that $S=PSL(2,q)$.   So $H=PSL(2,q)$
  or $PGL(2,q)$.   If $x \in G \setminus{H}$, then $x$ can be taken
  to be a field automorphism of order $e \ge 2$.

  If $H=PGL(2,q)$, the number of stable  semisimple classes
  is $q^{1/e} + 1$ and there is one (stable) unipotent class
  (the stable classes are precisely those of $PGL(2,q^{1/e})$).
  If $H=PSL(2,q)$, there are even fewer stable classes.
  Thus, the number of conjugacy classes in the coset $xH$
  is at most $q^{1/e} + 2$, whence the result holds.

So we may assume that $r > 1$.

Let $x \in G \setminus{H}$.  By the previous result,
the number of $x$-stable classes in $H$ is at most $Eq^{r-1}$
for some universal constant $E$.   Thus by  Lemma \ref{weak shintani},
the number of conjugacy classes in $xH$ is at most
$Eq^{r-1}$.    Since $[G:H] \le  6 \log q$,  it follows that
$$
k(G) \le   q^r  + 6E (\log q)  q^{r-1},
$$
and the result follows.
\end{proof}

With a bit more effort, one can remove the $\log q$ factor in the previous
result. Keeping track of the constants in the proof above gives Corollary
\ref{B}.

\section{Minimum Centralizer Sizes for the Finite Classical Groups}
\label{mincent}

This section gives lower bounds on centralizer sizes in finite classical
groups, and hence upper bounds on the size of the largest conjugacy class
in a finite classical group. Formulas for the conjugacy classes sizes go
back to Wall \cite{W}, but being quite complicated polynomials in $q$
effort is required to give explicit bounds. The bounds presented here hold
for all values of $n$ and $q$ and are also applied in \cite{FG2} and
\cite{Sh}.

The following standard notation about partitions will be used. Let
$\lambda$ be a partition of some non-negative integer $|\lambda|$ into
parts $\lambda_1 \geq \lambda_2 \geq \cdots$. Let $m_i(\lambda)$ be the
number of parts of $\lambda$ of size $i$, and let $\lambda'$ be the
partition dual to $\lambda$ in the sense that $\lambda_i' = m_i(\lambda) +
m_{i+1}(\lambda) + \cdots$. It is also useful to define the diagram
associated to $\lambda$ by placing $\lambda_i$ boxes in the $i$th row. We
use the convention that the row index $i$ increases as one goes downward.
So the diagram of the partition $(5441)$ is

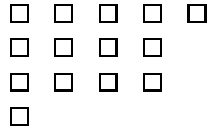
\[ \begin{array}{c c c c c} \framebox{} & \framebox{} & \framebox{}
& \framebox{} &  \framebox{} \\ \framebox{} &  \framebox{}& \framebox{} &
\framebox{} & \\
\framebox{} & \framebox{} & \framebox{} & \framebox{}&    \\ \framebox{} &
& & &
\end{array}
\] and $\lambda_i'$ can be interpreted as the size of the $i$th column.
The notation $(u)_m$ will denote $(1-u)(1-u/q) \cdots (1-u/q^{m-1})$.
This section freely uses Lemma \ref{pent} from Section \ref{boundnumber}.

\subsection{The general linear groups}

The following result about partitions will be helpful.

\begin{lemma} \label{monpartition} Let $\lambda$ be any partition. Then
for $q \geq 2$, \[ q^{\sum_i (\lambda_i')^2} \prod_i
(1/q)_{m_i(\lambda)} \geq q^{|\lambda|}(1-1/q).\] \end{lemma}

\begin{proof} Define a function $f$ on partitions by $f(\lambda)=
q^{\sum_i (\lambda_i')^2} \prod_i (1/q)_{m_i(\lambda)}$. Let $\tau$ be a
partition obtained from $\lambda$ by moving a box from a row of length $i$
to a row of length $j \geq i+1$. The idea is to show that $f(\lambda) \geq
f(\tau)$. The result then follows because a sequence of such moves
transforms any partition into the one-row partition and $f$ evaluated on
this partition is $q^{|\lambda|}(1-1/q)$.

One checks that $\sum_i (\lambda_i')^2 - \sum_i (\tau_i')^2 \geq 2$, so if
$i>1$, then \begin{eqnarray*} \frac{f(\tau)}{f(\lambda)} & \leq &
\frac{(1/q)_{m_{i-1}(\lambda)+1} (1/q)_{m_i(\lambda)-1}
(1/q)_{m_j(\lambda)-1} (1/q)_{m_{j+1}(\lambda)+1}} {q^2
(1/q)_{m_{i-1}(\lambda)} (1/q)_{m_i(\lambda)} (1/q)_{m_j(\lambda)}
(1/q)_{m_{j+1}(\lambda)}} \\
& = & \frac{(1-1/q^{m_{i-1}(\lambda)+1})(1-1/q^{m_{j+1}(\lambda)+1})}{q^2
(1-1/q^{m_{i}(\lambda)})(1-1/q^{m_{j}(\lambda)})}\\
& \leq & \frac{1}{q^2(1-1/q)^2} \leq 1,\end{eqnarray*} as desired. For the
$i=1$ case, the only difference in the above argument is that the term
$(1-1/q^{m_{i-1}(\lambda)+1})$ does not appear.
\end{proof}

Lemma \ref{countirr} is well-known and is proved by counting the non-zero
elements in a degree $r$ extension of $\F_q$ by the degrees of their
minimal polynomials.

\begin{lemma} \label{countirr} Let $N(q;d)$ be the number of monic
degree $d$ irreducible polynomials over the finite field $\F_q$,
disregarding the polynomial $z$. Then \[\sum_{d|r} dN(q;d) = q^r-1.\]
\end{lemma}

Lemma \ref{crudebound} will also be needed.

\begin{lemma} \label{crudebound} For $s \geq 2$, $(1-1/s)^s \geq
e^{-(1+1/s)}$ \end{lemma}

\begin{proof} \begin{eqnarray*} \log(1-1/s)^s & = &
-1-\frac{1}{2s}-\frac{1}{3s^2}-\cdots\\ & \geq &
-1-\frac{1}{2s}-\frac{1}{2s^2}-\cdots\\ & = & -1-\frac{1}{2s(1-1/s)}\\ &
\geq & -1-1/s. \end{eqnarray*} where the final inequality uses that $s
\geq 2$. \end{proof}

\begin{theorem} \label{GLsmallcent} The smallest centralizer size of an
element of $GL(n,q)$ is at least $\frac{q^n(1-1/q)}{e(1+\log_q(n+1))}$.
\end{theorem}

\begin{proof} It is well known that conjugacy classes of $GL(n,q)$
are parametrized by Jordan canonical form. That is for each monic
irreducible polynomial $\phi \neq z$, one picks a partition
$\lambda(\phi)$ subject to the constraint $\sum_{\phi} deg(\phi)
|\lambda(\phi)| = n$. The corresponding centralizers sizes are well known
(\cite{Mac}, page 181) and can be rewritten as \[ \prod_{\phi}
q^{deg(\phi) \sum_i (\lambda(\phi)_i')^2} \prod_i
(1/q^{deg(\phi)})_{m_i(\lambda(\phi))}.\] To minimize this expression,
suppose that for each polynomial $\phi$ one knows the size
$|\lambda(\phi)|$ of $\lambda(\phi)$. Lemma \ref{monpartition} shows that
$\lambda(\phi)$ should be taken to be a one-row partition which would
contribute $q^{|\lambda(\phi)| \cdot deg(\phi)} (1-1/q^{deg(\phi)})$.
Letting $r$ be such that $\sum_{i=1}^r i N(q;i) \geq n$, it follows that
the minimal centralizer size is at least $q^n \prod_{i=1}^r
(1-1/q^i)^{N(q;i)}$.

Observe that $r$ can be taken to be the smallest integer such that
$q^r-1 \geq n$, because by Lemma \ref{countirr}

\[ \sum_{i=1}^r i N(q;i) \geq \sum_{d|r} d N(q;d) = q^r-1.\] Since
$N(q;i) \leq q^i/i$, the minimum centralizer size is at least $q^n
\prod_{i=1}^r (1-1/q^i)^{q^i/i}$. Lemma \ref{crudebound} gives that the
minimum centralizer size is at least $\frac{q^n (1-1/q)}{e^{
(1+1/2+\cdots+1/r)}}$. To finish the proof use the bounds
$1+1/2+\cdots+1/r \leq 1+ \log(r)$ and take $r=1+\log_q(n+1)$. \end{proof}

{\it Remark:} Since the number of conjugacy classes of $GL(n,q)$ is less
than $q^n$, one might hope that the largest conjugacy class size is at
most $\frac{|GL(n,q)|}{c q^n}$ where $c$ is a constant. The proof of
Theorem \ref{GLsmallcent} shows this to be untrue. Indeed, the infinite
product $\prod_i (1-1/q^i)^{N(q;i)}$ vanishes. This can be seen by setting
$u=1/q$ in the identity \[ \prod_{i=1}^{\infty} (1-u^i)^{-N(q;i)} = 1+
\frac{q-1}{q} \sum_{n \geq 1} u^nq^n \] (which holds since the coefficient
of $u^n$ on both sides counts the number of monic degree $n$ polynomials
with non-vanishing constant term).

\subsection{The unitary groups}

The method for the finite unitary groups is similar to that for $GL(n,q)$.
As usual, we view $U(n,q)$ as a subgroup of $GL(n,q^2)$.

\begin{lemma} \label{monpartition2} Let $\lambda$ be any partition. Then
for $q \geq 2$, \[ q^{\sum_i (\lambda_i')^2} \prod_i
(-1/q)_{m_i(\lambda)} \geq q^{|\lambda|}(1+1/q).\] \end{lemma}

\begin{proof} The argument is the same as for Lemma \ref{monpartition}. Using
the same notation as in that proof, except that now $f(\lambda)= q^{\sum_i
(\lambda_i')^2} \prod_i (-1/q)_{m_i(\lambda)}$, one obtains that
\begin{eqnarray*} \frac{f(\tau)}{f(\lambda)} & \leq &
\frac{(1-(-1/q)^{m_{i-1}(\lambda)+1})(1-(-1/q)^{m_{j+1}(\lambda)+1})}{q^2
(1-(-1/q)^{m_{i}(\lambda)})(1-(-1/q)^{m_{j}(\lambda)})}\\
& \leq & \frac{(1+1/q)^2}{q^2 (1-1/q^2)^2} \leq 1.\end{eqnarray*}
\end{proof}

Given a polynomial $\phi$ with coefficients in $\F_{q^2}$ and
non-vanishing constant term, define a polynomial $\tilde{\phi}$ by \[
\tilde{\phi} = \frac{z^{deg(\phi)} \phi^q(\frac{1}{z})}{[\phi(0)]^q} \]
where $\phi^q$ raises each coefficient of $\phi$ to the $q$th power. A
polynomial $\phi$ is called self-conjugate if $\tilde{\phi}=\phi$ and an
element in an extension field of $\F_{q^2}$ is called self-conjugate if
its minimal polynomial over $\F_{q^2}$ is self-conjugate.

\begin{lemma} \label{countemunit} Suppose that $r$ is odd. Then the
number of nonzero non-self-conjugate elements in $\F_{q^{2r}}$ viewed as
an extension of $\F_{q^2}$ is $q^{2r}-q^r-2$.
\end{lemma}

\begin{proof} Theorem 9 of \cite{F2} shows that the number of
self-conjugate elements of degree $i$ over $\F_{q^2}$ is $0$ if $i$ is
even and is $\sum_{d|i} \mu(d)(q^{i/d}+1)$ if $i$ is odd, where $\mu$ is
the Moebius function. Thus Moebius inversion implies that the total number
of self-conjugate elements of $\F_{q^{2r}}$ is \[ \sum_{i|r} \sum_{d|i}
\mu(d) (q^{i/d}+1) = q^r+1,\] which implies the result.
\end{proof}

\begin{theorem} \label{Usmallcent} The smallest centralizer size of an
element of $U(n,q)$ is at least $q^n
\left(\frac{1-1/q^2}{e(2+\log_q(n+1))} \right)^{1/2}$. \end{theorem}

\begin{proof} For $n=1$ this is clear so suppose that $n>1$. The
conjugacy classes of $U(n,q)$ and their sizes were determined in
\cite{W}. They are parametrized by the following analog of Jordan
canonical form. For each monic irreducible polynomial $\phi \neq z$, one
picks a partition $\lambda(\phi)$ subject to the two constraints that
$\sum_{\phi} deg(\phi) |\lambda(\phi)| = n$ and
$\lambda(\phi)=\lambda(\tilde{\phi})$. The corresponding centralizers
sizes are due to Wall and can be usefully rewritten as

\[ \prod_{\phi \neq z, \phi = \tilde{\phi}} q^{deg(\phi) \sum_i
(\lambda(\phi)_i')^2} \prod_i (-1/q^{deg(\phi)})_{m_i(\lambda(\phi))} \]

\[ \cdot \left( \prod_{\{\phi,\tilde{\phi}\}, \phi \neq \tilde{\phi}}
q^{deg(\phi) \sum_i (\lambda(\phi)_i')^2} \prod_i
(1/q^{deg(\phi)})_{m_i(\lambda(\phi))} \right)_{q \mapsto q^2}. \] Here
the $q \mapsto q^2$ means (in the second product over polynomials) to
replace all occurrences of $q$ by $q^2$. Note that the second product is
over unordered conjugate pairs of non self-conjugate monic irreducible
polynomials.

Note that the bound in Lemma \ref{monpartition2} is greater than
$q^{|\lambda|}$ whereas the bound in Lemma \ref{monpartition} is less than
$q^{|\lambda|}$. Hence the minimum size centralizer will correspond to a
conjugacy class whose characteristic polynomial has only non
self-conjugate irreducible polynomials as factors. Let $\tilde{M}(q;i)$
denote the number of unordered pairs $\{\phi,\tilde{\phi}\}$ where $\phi$
is monic non self-conjugate and irreducible of degree $i$ with
coefficients in $F_{q^2}$. Then a lower bound for the smallest centralizer
size is $q^n \prod_{i=1}^r (1-1/q^{2i})^{\tilde{M}(q;i)}$ where $r$ is
such that $\sum_{i=1}^r 2i \tilde{M}(q;i) \geq n$. Take $r$ to be odd, and
observe by Lemma \ref{countemunit} that \[ \sum_{i=1}^r 2i \tilde{M}(q;i)
\geq \sum_{i|r} 2i \tilde{M}(q;i) = q^{2r}-q^r-2, \] and that
$q^{2r}-q^r-2 \geq n$ if $q^r \geq n+1$ (since $n>1$). Since
$\tilde{M}(q;i) \leq \frac{q^{2i}}{2i}$, the smallest centralizer size is
at least $q^n \prod_{i=1}^r (1-1/q^{2i})^{\frac{q^{2i}}{2i}}$. Arguing as
in the general linear case and applying Lemma \ref{crudebound} proves the
theorem.
\end{proof}

\subsection{Symplectic and orthogonal groups}

    To begin the study of minimum centralizer sizes in symplectic
and orthogonal groups, we treat the case of elements whose
characteristic polynomial is $(z \pm 1)^n$. This will be done by two
different methods. The first approach uses algebraic group techniques
and gives the best bounds. The second approach is combinatorial but of
interest as it involves a new enumeration of unipotent elements in
orthogonal groups. Note that there is no need to consider odd
dimensional orthogonal groups in even characteristic, as these are
isomorphic to symplectic groups.

\begin{prop} \label{centsize}
\begin{enumerate}
\item The minimum
centralizer size of an element in the group $Sp(2n,q)$ with
characteristic polynomial $(z \pm 1)^{2n}$ is at least $q^n$.
\item In odd
characteristic, the minimum centralizer size of an element with
characteristic polynomial $(z \pm 1)^{2n}$ in $O^{\pm}(2n,q)$ or $(z \pm
1)^{2n+1}$ in $O(2n+1,q)$ is at least $q^n$.
\item In even characteristic, the
minimum centralizer size of a unipotent element of $O^{\pm}(2n,q)$ is at
least $2q^{n-1}$.
\end{enumerate}
\end{prop}

\begin{proof}  We first work in the
ambient algebraic group $G$ and connected component $H$.  Let $g \in G(q)$.
 Let $B$
be a Borel subgroup of $H$ normalized by $g$ and $U$ its unipotent radical.
First suppose that $g \in H$.
  It suffices to show that $C_U(g)$ has dimension at least $r$, the rank of $H$.
  For then the rational points in $C_U(g)$ have order a multiple of $q^r$
  as required (cf \cite{FG1}).  The centralizer in $B$ of a regular unipotent element in $B$
  has dimension exactly $r$ in $B$. Since these elements are dense in $U$, the same
  is true for any such element.

  Now suppose that $g$ is not in $H$.  This only occurs in even characteristic
  with $G$   an orthogonal group.
    Now $G$ embeds in a symplectic group $L$ of the same dimension.
  Let $A$ be a maximal unipotent subgroup of $L$ containing $U$.  Note that
  $gU$ contains regular unipotent elements of $L$ and so the subset of $gU$
  consisting of regular unipotent elements is dense in $gU$.

  We claim
  that $C_U(g)$ has dimension at least $r-1$.  Once we have established that claim,
  it follows as above that the centralizer in $H(q)$ of $g$ is divisible by $q^{r-1}$
  as required.   Since $g$ is not in $H$,   its centralizer in $G(q)$
  has order at least twice as large.
  Since the regular unipotent elements in $gU$ are dense, it suffices
  to prove the claim for such an element.

  Since $g$ is a regular unipotent element of $L$,  it
  follows that  $C_L(g) = C_A(g)$ has dimension $r$.  On the other hand, we see also that
  the set of orthogonal groups containing $g$ is a $1$-dimensional variety
  (the orthogonal groups containing $g$ are in bijection with the  $g$-invariant hyperplanes
  of the orthogonal module for $L$ that do not contain the $L$ fixed space -- since
  any such hyperplane must contain the image of $g-1$ which has codimension $2$,
  we see the set of hyperplanes is a $1$-dimensional variety).  Now $C_U(g)$ is precisely
  the stabilizer of the hyperplane corresponding to $H$ and so has codimension
  at most $1$ in $C_A(g)$.  Thus, $\dim C_A(g) \ge r-1$ (in fact equality holds).
  This proves the claim and completes the proof.
  \end{proof}

  We now give a more combinatorial approach to lower bounding
centralizer sizes of elements whose characteristic polynomial is $(z \pm
1)^n$; this is complementary and yields different information. A crucial
step in this approach is counting the number of unipotent elements in
symplectic and orthogonal groups. Steinberg (see \cite{C} for a proof)
showed that if $G$ is a connected reductive group and $F:G \mapsto G$ is a
Frobenius map, then the number of unipotent elements of $G^F$ is the
square of the order of a p-Sylow, where p is the characteristic. We remind
the reader that
\begin{enumerate}
\item $|Sp(2n,q)|=q^{n^2} \prod_{j=1}^n (q^{2j}-1)$.
\item $|O(2n+1,q)|=2q^{n^2} \prod_{j=1}^n (q^{2j}-1)$ (in odd characteristic).
\item $|O^{\pm}(2n,q)|=2q^{n^2-n} (q^n \mp 1) \prod_{j=1}^{n-1} (q^{2j}-1)$. \end{enumerate}
This implies that the number of unipotent elements in $Sp(2n,q)$ is $q^{2n^2}$. However the
orthogonal groups are not connected, so Steinberg's theorem is not directly applicable.
Nevertheless, in odd characteristic, unipotent elements always live in $\Omega$, so
Steinberg's theorem does imply that the number of unipotent elements in $O(2n+1,q)$ in
odd characteristic is $q^{2n^2}$, and that the number of unipotent elements of
$O^{\pm}(2n,q)$ in odd characteristic is $q^{2(n^2-n)}$.

Proposition \ref{unipenum} uses generating functions to treat orthogonal
groups in even characteristic; along the way we obtain a formula for the
number of unipotent elements (this turns out not to be a power of $q$ and
seems challenging from the algebraic approach). Two combinatorial lemmas
are needed.

\begin{lemma} \label{Euler} (Euler) \cite{A1}
\[ \prod_{j=1}^{\infty} (\frac{1}{1-\frac{u}{q^j}}) =
\sum_{n \geq 0} \frac{u^n q^{{n \choose 2}}}{(q^n-1) \cdots (q-1)}.\]
\end{lemma}

    To state the second lemma, we require some notation (which will be used elsewhere
in this subsection as well). Given a a polynomial $\phi(z)$ with
coefficients in $\F_q$ and non vanishing constant term, define the
``conjugate'' polynomial $\phi^*$ by \[ \phi^* = \frac{z^{deg(\phi)}
  \phi(\frac{1}{z})}{\phi(0)}. \]
One calls $\phi$ self-conjugate if $\phi^*=\phi$. Note that
the map $\phi \mapsto \phi^*$ is an involution. We let $N^*(q;d)$ denote the
 number of monic irreducible self-conjugate polynomials of degree $d$
 with coefficients in $\F_q$, and let $M^*(q;d)$ denote the number of
 conjugate pairs of monic irreducible non-self conjugate polynomials
 of degree $d$ with coefficients in $\F_q$.

\begin{lemma} \label{fnplem} (\cite{FNP}) Let $f=1$ if the characteristic is even
 and $f=2$ if the characteristic is odd. Then
\[ \prod_{d \geq 1} (1-t^d)^{-N^*(q;2d)} (1-t^d)^{-M^*(q;d)} = \frac{(1-t)^f}{1-qt} \]
\[ \prod_{d \geq 1} (1+t^d)^{-N^*(q;2d)} (1-t^d)^{-M^*(q;d)} = 1-t. \] \end{lemma}

    Now we can enumerate unipotent elements in even characteristic
    orthogonal groups.

\begin{prop} \label{unipenum} Suppose that the characteristic is even.
Then the number of unipotent elements of $O^{\pm}(2n,q)$ is
$q^{2n^2-2n+1}(1+\frac{1}{q} \mp \frac{1}{q^n})$.
\end{prop}

\begin{proof} Given a group $G$, we let $u(G)$ denote the proportion of
elements of $G$ which are unipotent. We define generating functions $F^{\pm}(t)$ by
\[ F^{\pm}(t) = 1 + \sum_{n \geq 1} t^{n} \left( u(O^+(2n,q)) \pm u(O^-(2n,q)) \right). \]
There is a notion of cycle index for the orthogonal groups
(see \cite{F1} or \cite{F2} for background),
and the cycle indices for the sum and difference of the orthogonal groups factor.
Setting all variables equal to 1 in the cycle index for the sum of $O^+(n,q)$ and
$O^-(n,q)$, it follows that \[ \frac{1+t}{1-t} =
F^+(t) \frac{\prod_{j \geq 1} (1-\frac{t}{q^{2j-1}})}{1-t}.\] Let
 us make some comments about this equation. Here the $F^+(t)$
 corresponds to the part of the cycle index for the polynomial
 $z-1$. The term $\frac{\prod_{j \geq 1}
 (1-\frac{t}{q^{2j-1}})}{1-t}$ corresponds to the remaining possible
 factors of the characteristic polynomial. This follows from the
 combinatorial identity \[ \prod_{d \geq 1} \prod_{r \geq 1}
 \left( 1+(-1)^r\frac{t^{d}}{q^{dr}} \right)^{-N^*(q;2d)}
 \left( 1-\frac{t^{d}}{q^{dr}} \right)^{-M^*(q;d)} = \frac{\prod_{j \geq 1}
 (1-\frac{t}{q^{2j-1}})}{1-t}\] which is a consequence of Lemma
 \ref{fnplem} after reversing the order of the products. Solving for
 $F^+(t)$, one finds that \[F^+(t) =
 \frac{1+t}{\prod_{j \geq 1} (1-\frac{t}{q^{2j-1}})}.\] Taking the
 coefficient of $t^n$ and using Lemma \ref{Euler},
 it follows that \begin{eqnarray*} & & u(O^+(2n,q)) + u(O^-(2n,q))\\ &
 = & \frac{1}{q^{n-1}(1-1/q^2) \cdots
 (1-1/q^{2n-2})} \left(1 + \frac{1}{q(1-1/q^{2n})} \right).
 \end{eqnarray*}

    Next we solve for $F^-(t)$. Setting all variables equal to 1
    in the cycle index for the difference of $O^+(n,q)$ and $O^-(n,q)$,
    it follows that
\[ 1 = F^-(t) \prod_{j \ even} (1-\frac{t}{q^j}).\] Here the $F^-(t)$
 corresponds to the part of the cycle index for the polynomial
 $z-1$. The other term on the right hand side corresponds to the
 remaining possible factors of the characteristic polynomial. This
 follows from the combinatorial identity \[ \prod_{d \geq 1} \prod_{r
 \geq 1} \left( 1-(-1)^r \frac{t^{d}}{q^{dr}} \right)^{-N^*(q;2d)}
 \left( 1-\frac{t^{d}}{q^{dr}} \right)^{-M^*(q;d)} = \prod_{j \ even}
 (1-\frac{t}{q^j}) \] which is a consequence of Lemma \ref{fnplem}
 after reversing the order of the products. Thus $F^-(t) = \prod_{j \
 even} (1-\frac{t}{q^j})^{-1}$. Taking the coefficient of $t^n$ and
 using Lemma \ref{Euler}, it follows that \[ u(O^+(2n,q)) -
 u(O^-(2n,q)) = \frac{1}{q^{2n}(1-1/q^2) \cdots (1-1/q^{2n})}.\]

    Having found formulas for $u(O^+(2n,q)) + u(O^-(2n,q))$ and
    $ u(O^+(2n,q)) - u(O^-(2n,q))$ one now solves for $u(O^{\pm}(2n,q))$ giving the
    statement of the proposition. \end{proof}

    Proposition \ref{unipcentcomb} gives lower bounds on centralizer sizes for elements
    in symplectic and orthogonal groups whose characteristic polynomial is
     $(z \pm 1)^n$. Note that the bound of Proposition \ref{centsize} was only slightly stronger.

\begin{prop} \label{unipcentcomb}
\begin{enumerate}
\item The centralizer size of an element in the group $Sp(2n,q)$ whose characteristic
 polynomial has only factors $(z \pm 1)^{2n}$ is at least $q^n (1-\frac{1}{q^2}-\frac{1}{q^4})$.

\item In odd characteristic, the centralizer size of an element with characteristic
polynomial $(z \pm 1)^{2n+1}$ in $O(2n+1,q)$ or $(z \pm 1)^{2n}$ in
$O^{\pm}(2n,q)$ is at least $q^n$.
\item In even characteristic, the centralizer size of a unipotent element in
$O^{\pm}(2n,q)$ is at least $q^{n-1}(1-1/q^2-1/q^4)$.
\end{enumerate}
\end{prop}

\begin{proof} For the first assertion, suppose without loss of generality that
the element is unipotent. By Steinberg's theorem the total number of
unipotent elements in $Sp(2n,q)$ is $q^{2n^2}$. Hence the sum of the
reciprocals of the centralizer sizes of unipotent elements is equal to
$\frac{q^{2n^2}}{|Sp(2n,q)|}$, from which it follows that the centralizer
size of any unipotent element is at least \[ \frac{|Sp(2n,q)|}{q^{2n^2}} =
q^n (1-1/q^2)
    \cdots (1-1/q^{2n}) \geq q^n (1-1/q^2-1/q^4). \] Note that the final inequality is
    Lemma \ref{pent}.

    For the second assertion, suppose without loss of generality that
    the element is unipotent.
By Steinberg's theorem the total number
 of unipotent elements in $O(2n+1,q)$ is $q^{2n^2}$. Thus the
 centralizer size of any unipotent conjugacy class of $O(2n+1,q)$ is
 at least \[ \frac{|O(2n+1,q)|}{q^{2n^2}} = 2q^n (1-1/q^2) \cdots
 (1-1/q^{2n}) \geq q^n \] where the inequality uses Lemma \ref{pent}
 and the fact that $q \geq 3$. Similarly, in the even dimensional case
 with $q$ odd, the centralizer size is at least
\[ \frac{|O^{\pm}(2n,q)|}{q^{2(n^2-n)}} = 2q^n (1-1/q^2)
\cdots (1-1/q^{2n-2})(1 \mp 1/q^n) \geq q^n. \]

In part 3 the characteristic is even, and using the count of unipotent
    elements in Proposition \ref{unipenum}, it follows that the centralizer size of any
    unipotent element of $O^{\pm}(2n,q)$ is at least \begin{eqnarray*} & &
    \frac{|O^{\pm}(2n,q)|} {q^{2n^2-2n+1}(1+\frac{1}{q} \mp \frac{1}{q^n})}\\
& = & \frac{2q^{n-1} (1-1/q^2) \cdots (1-1/q^{2n-2}) (1 \mp 1/q^n)}{(1+1/q \mp 1/q^n)}\\
& \geq & q^{n-1} (1-1/q^2) \cdots (1-1/q^{2n-2})\\
& \geq & q^{n-1}(1-1/q^2-1/q^4). \end{eqnarray*} \end{proof}

Theorem \ref{answer} is the main result of this subsection. Note that
there is no need to consider odd dimensional even characteristic
orthogonal groups, as these are isomorphic to symplectic groups.

\begin{theorem} \label{answer} \begin{enumerate}
\item The
centralizer size of an element of $Sp(2n,q)$ is at least \[q^n \left[
\frac{1-1/q}{2e (\log_q(4n)+4)} \right]^{1/2}.\]
\item The
centralizer size of an element of $O^{\pm}(2n,q)$ is at least \[ 2q^{n-1}
\left[ \frac{1-1/q}{2e (\log_q(4n)+4)} \right]^{1/2}.\]
\item The
centralizer size of an element of $SO^{\pm}(2n,q)$ is at least \[ q^{n}
\left[ \frac{1-1/q}{2e (\log_q(4n)+4)} \right]^{1/2}.\]
\item Suppose that $q$
is odd. The centralizer size of an element of $O(2n+1,q)$ is at least
\[q^n \left[ \frac{1-1/q}{2e (\log_q(4n)+4)} \right]^{1/2}.\]
\end{enumerate}
\end{theorem}

\begin{proof} First consider the case $Sp(2n,q)$.
Wall \cite{W} parametrized the conjugacy classes of $Sp(2n,q)$ and found
their centralizer sizes. As in the general linear and unitary cases, the
formula is multiplicative with terms coming from self-conjugate
irreducible polynomials and also conjugate pairs of non-self conjugate
irreducible polynomials. By Lemma \ref{centsize} a size $k$ partition
corresponding to a polynomial $z-1$ or $z+1$ contributes at least a factor
of $q^{k/2}$. As with the unitary groups, one sees that a partition
$\lambda$ from a self-conjugate irreducible polynomial $\phi$ contributes
at least $q^{deg(\phi) \cdot |\lambda|/2}$ and that $\lambda$ associated
with a pair $\{\phi,\bar{\phi} \}$ with $\phi$ monic non self-conjugate
irreducible contributes at least $q^{deg(\phi) \cdot |\lambda|}
(1-1/q^{deg(\phi)})$. Then it follows that a lower bound for the smallest
centralizer size is $q^n \prod_{i=1}^{2^r} (1-1/q^i)^{M^*(q;i)}$ where
$2^r$ is chosen such that $\sum_{i=1}^{2^r} 2i M^*(q;i) \geq 2n$. From
\cite{FNP}, if $r \geq 1$ then \begin{eqnarray*} 2^{r+1} M^*(q;2^r) & = &
2^r N(q;2^r)-2^r N^*(q;2^r)\\
& = & \left( q^{2^r}-q^{2^{r-1}} \right) - 2^r N^*(q;2^r) \geq
q^{2^r}-2q^{2^{r-1}}. \end{eqnarray*} Note that if $r \geq 2$, then
$q^{2^r}-2q^{2^{r-1}} \geq \frac{q^{2^r}}{2}$. It follows that if $r \geq
2$ and $q^{2^r} \geq 4n$, then $\sum_{i=1}^{2^r} 2i M^*(q;i) \geq 2n$.
Thus we need a $2^r$ which is at least $\max\{ 4,\log_q(4n) \}$, and one
can find such a $2^r$ which is at most $2(\log_q(4n)+4)$. Since $M^*(q;i)
\leq q^i/2i$, arguing as in the general linear case proves the first
assertion of the theorem.

    For the remaining assertions the contribution to the centralizer size
coming from the part of the characteristic polynomial relatively prime to
$z^2-1$ is the same for symplectic and orthogonal groups. Thus it is
sufficient to focus on the part of the characteristic polynomial of the
form $(z-1)^a (z+1)^b$ where $b=0$ if the characteristic is even. For
$O^{\pm}(2n,q)$,  the contribution must be at least
$q^{\frac{a+b}{2}-1}$ -- for either the characteristic is odd and $a,b$
have the same parity and   part 2 of Proposition \ref{centsize} applies, or else
the characteristic is even  part 3 of Proposition \ref{centsize} applies.
Note that if the element is not in $SO$, then the centralizer size is doubled
in $0$ giving (3).   If the element is in $SO$, then $a$ and $b$ are even.
Thus, arguing as above, the minimum centralizer
size is $q^{\frac{a+b}{2}}$ since $a$ and $b$ are both even, and (3) follows.
For $O^{\pm}(2n+1,q)$ the contribution must be at least $
q^{\frac{a+b-1}{2}}$ since $a,b$ have unequal parity,  and use part 2 of
Proposition \ref{centsize}.
 \end{proof}

 For exceptional groups (or more generally for groups of bounded rank),
 we have:

 \begin{lemma} \label{except-centralizer}  Let $G$ be a connected
 simple exceptional algebraic group with $F$ a Frobenius endomorphism
 associated to the field of $q$ elements.     If $g \in G^F$, then
 $|C_{G^F}(g)| \ge q^r/26$.
 \end{lemma}

 \begin{proof}   Since $\dim C_G(g) \ge r$, it follows that
 $|C_{G^F}(g)| \ge (q-1)^r$.   Since $r \le 8$, the result follows
 for $q > 2$.  If $q=2$, the result follows by inspection (see \cite{Lu}).
 \end{proof}

Note that Theorem \ref{DD}
is an immediate consequence of
the  results in  this section.   In applications, we will  have to deal with
simple groups as well, so we state this  (in all cases except for type A,
the index of the simple group in the group of inner diagonal automorphisms
is at most $4$  -- in type $A$, we use the results for $GL(n,q)$ or $U(n,q)$
and divide by $(q \mp 1)\gcd(q \mp 1, n)$ -- one factor to pass to $SL$ or
$SU$ and the other factor for the center of these groups).   Thus, we have:

\begin{theorem} \label{D}   Let $S$ be a simple Chevalley group defined
over the field of $q$ elements with $r$ the rank of the ambient algebraic
group.   There is a universal constant $A$ such that if $x$ is an inner diagonal
automorphism of $S$, then
$$
|C_S(x)| \ge   \frac{q^r}{A (\min\{q,r\})(1 + \log_q(r))} \ge \frac{q^{r-1}}{A(1 + \log_q(r))}.
$$
\end{theorem}

The result  also applies to the full orthogonal
group as well.  The result fails for other graph automorphisms and
field automorphisms.

\section{Conjugacy Classes of Maximal Subgroups and  Derangements}  \label{boston-shalev}

We want to obtain bounds on the number of conjugacy classes of
maximal subgroups of finite simple groups.   We will then combine
these results with our results on class numbers to  obtain very
strong  results on  the proportion of derangements in actions
of simple and almost simple  groups.  We define $m(G)$
to be the number of conjugacy classes of maximal subgroups of $G$.
Aschbacher and the second author \cite{AG} conjectured that
$m(G) < k(G)$, and proved this for $G$ solvable.
Note that if $G$ is an elementary abelian $2$-group,
$m(G)+1 = k(G)=|G|$.

A related conjecture (of Wall)
is that the number of maximal subgroups of a finite group $G$
is less than $|G|$.  Wall proved this for solvable groups.
See \cite{LPS} for more recent results.

First we  note  the following result, which we will not require
--- see \cite{LS3} and combine this
with \cite{LMS}.

\begin{lemma} \label{alt-max}  If $G=A_n$ or $S_n$,
then $m(G) \le n^{1 + o(1)}$.
\end{lemma}

It follows by  \cite{LMS} that:

\begin{theorem} \label{boundedrank-max}
Let $G$ be an almost simple Chevalley group
of rank $r$ defined over the field  of $q$ elements.
 Then  $m(G) \le c(r) + 2r \log\log q$.
\end{theorem}

The $\log\log q$ term comes from subfield  groups.

This immediately  gives
a generalization of the
Boston-Shalev conjecture in the case of bounded rank.

\begin{theorem}   \label{bstronger}   Let  $G$ be an almost simple
group with socle $S$ a Chevalley group of  fixed rank $r$
defined over $\F_q$.   Assume that  $G$ is  contained
in  the group of inner-diagonal  automorphisms of  $S$.
Let $\mathcal{M}(G)$ denote the set of maximal  subgroups
of $G$ that do not  contain a maximal  torus of $S$.
Then
$$
\lim_{q \rightarrow \infty}  \frac{|\cup_{M \in \mathcal{M}(G)} M|}{|G|} =
0.
$$
\end{theorem}

\begin{proof}   It follows by the basic results about maximal
subgroups of $G$ (cf \cite{FG1})  and Corollary \ref{B} that any maximal subgroup
$M$ of $G$ either
contains  a maximal torus  of $S$ or satisfies
$k(M) < Cq^{r-1}$ with $C$ a universal constant.
Applying Theorem \ref{DD}  (with $r$ fixed)  gives that
for any maximal subgroup $M$ of $G$ not containing a maximal torus,
$$
 |\cup_{g \in G} M^g |/|G| \leq k(M) \cdot \max_{x \in G} \frac{1}{|C_G(x)|}  <O(1/q).
 $$
Thus, by  Theorem \ref{boundedrank-max} and the fact that $r$ is fixed,
$$
 \frac{|\cup_{M \in \mathcal{M}(G)} M|}{|G|}  < O\big(\frac{ \log\log q}{q}\big),
 $$
 whence the result.   Indeed,  separating out the subfield case shows
 that $O(1/q)$ is an upper bound in the equation above.
\end{proof}

As in \cite{FG1}, this gives:

\begin{cor}  Let  $G$ be an almost simple
group with socle $S$ a Chevalley group of  fixed rank $r$
defined over $\F_q$.   Assume that  $G$ is  contained
in  the group of inner-diagonal  automorphisms of  $S$.
Let $M$ be a maximal subgroup of $G$ not  containing
$S$ and set $\Omega= G/M$.  Then
there exists a universal constant
 $\delta > 0$ such that $\delta(G,\Omega) > \delta$.
\end{cor}

{\it Remark:} An inspection  of the proof shows that the result holds for
the proportion of derangements in a given  coset of $S$.   This is no
longer true  if we allow field automorphisms.

\vspace{2mm}

We now want to consider what happens for increasing  $r$. In particular,
it suffices to consider  $r > 8$ and  so  we restrict our attention to
classical  groups.     The maximal subgroups have been classified by
Aschbacher  \cite{As} and we consider the families individually.

We recall Aschbacher's theorem on maximal subgroups of classical
Chevalley groups.   We refer the reader to the description of the subgroups in
\cite{As}.  See also \cite{KL}.

So let $G$ be a classical Chevalley group with natural module $V$ of
dimension $d$.  Then a subgroup $H$ of $G$ falls into the following nine
families.  In particular, a maximal subgroup is either in $\mathcal{S}$ or
is maximal in one of the families $\mathcal{C}_i$.  We will write
$\mathcal{C}_i(G)$ to denote the  maximal subgroups of $G$ that are in
the family $\mathcal{C}_i$. We let $\mathcal{S}(G)$ denote the maximal
subgroups of $G$ in $\mathcal{S}$.

\bigskip

\centerline
{{\sc Table 3} \quad Aschbacher Classes}
\begin{center}
\begin{tabular}{|c||c|} \hline
$\C_1$  &   $H$ preserves either a totally singular or
a nondegenerate subspace of $V$   \\  \hline
$\C_2$ &  $H$ preserves an additive decomposition of $V$ \\   \hline
$\C_3$ &  $H$ preserves an extension field structure of prime degree \\   \hline
$\C_4$ &  $V$ is tensor decomposable for $H$  \\  \hline
$\C_5$ &  $H$ is defined over a subfield of prime index \\  \hline
$\C_6$  & $d$ is a power of a prime and $H$ normalizes
a subgroup of symplectic type  \\   \hline
$\C_7$  & $V$ is tensor induced for $H$ \\   \hline
$\C_8$ &   $V$ is the natural module for a classical subgroup $H$; \\   \hline
$\mathcal{S}$ & $H$ is the normalizer of an almost simple group $S$
and $H$ is  not in $\C_i$ \\
\hline
\end{tabular}
\end{center}

\bigskip

It is easy to see (cf. \cite{GKS} and \cite[Lemmas 2.1, 2.4]{LPS}):

\begin{lemma} \label{tool1} The number of conjugacy classes of maximal subgroups
of $G$ in $ \cup_{H \in \mathcal{C}_i}$ is at most $8 r \log r  + r  \log \log q$.
\end{lemma}

It is straightforward to see using Corollary \ref{B} and the structure
of the maximal subgroups in $\mathcal{C}_i$ that:

\begin{lemma} \label{tool2} Let  $M \in \mathcal{C}_i(G)$ for $i > 3$.
There is a universal constant  $C$ such that
 $k(M) < Cq^{(r+1)/2}$.
\end{lemma}

The only subgroups that are close to that  bound are those in $\mathcal{C}_8$.

We will now prove:

\begin{theorem}  \label{stronger}
Let $G$ be a finite classical Chevalley group of rank $r$ over the field
of $q$ elements.   Let $X(G)$ denote the set of maximal  subgroups of $G$
contained  in  $\mathcal{S}(G) \cup_{i=4}^8   \mathcal{C}_i(G)$. For $r$
sufficiently large,
$$
\frac{|\cup_{H \in X(G)} H|}{|G|} < O(q^{-r/3}).
$$
\end{theorem}

We will prove this for each of the families, and taking unions implies the
result. We remark that the theorem applies to any subgroup $G$ between the
socle and the full isometry group of $V$. The result fails if we consider
almost simple groups with field automorphisms allowed (see \cite{GMS} for
examples of so called exceptional permutation groups and a classification
of primitive almost simple exceptional permutation actions).

Of course, a trivial corollary is Theorem \ref{C}.  Note also that our
estimate implies that the proportion of derangements in any
coset of the simple group tends to $1$ as well for the actions
considered in Theorem \ref{stronger}.

The idea of the proof is quite simple.   Let $X(G)$ denote
a set of subgroups of $G$ closed under conjugation.
We want to show that  the number of conjugacy classes
of $G$ that intersect some element of $X(G)$ is at most
$c(X)$.   Then using our results on a lower bound for centralizers
(or equivalently an upper bound for sizes of conjugacy classes),
we see that
$$
| \cup_{H \in X(G)} H| \le  \frac {c(X) A |G|(1 + \log_q r)}{q^{r-1}},
$$
or
$$
\frac{ | \cup_{H \in X(G)} H| }{|G|}  \le \frac {c(X) A (1 + \log_q r)}{q^{r-1}},
$$
where $A$ is a universal constant.
So we only need show that $c(X)$ is at most $O(q^{(2/3)r -2})$ in each
case.

We note that for a fixed simple group $S$, the number of embeddings of $S$
into $G$ is certainly bounded from above  by $2r k(\hat{S})$ where
$\hat{S}$ is the universal cover of $S$  (note that $k(\hat{S})$ is an
upper bound for the number of representations and the factor $2r$ comes
from the fact that we may have representations which are inequivalent in
the simple classical group but become conjugate in the full group of
isometries). The arguments vary slightly depending upon the family we are
considering but the basic idea is the same in all cases.

\begin{lemma}  \label{derangements-c38}
Let $G$ be a finite classical Chevalley group of rank
$r$ over the field of $q$ elements.   Let $X(G)$ denote the
set of maximal  subgroups of $G$  contained  in  $\mathcal{C}_i(G)$
for $i  > 3$.  Then  for $r$ sufficiently large,
$$
\frac{|\cup_{H \in X(G)} H|}{|G|} < O(q^{-r/3}).
$$
\end{lemma}


\begin{proof}  By Lemmas \ref{tool1} and \ref{tool2}, one knows that
$\cup_{H \in X(G)} H$ is the union of at most $ Cq^{(r+1)/2}(8 r \log r  +
r  \log \log q)$ conjugacy classes of $G$. By Theorem \ref{D}, each class
has size at most $A|G| (1 + \log_q(r))/q^{r-1}$, with $A$ a universal constant,
whence the result.
\end{proof}

We now  consider $\mathcal{S}(G)$.  It is convenient to
 split $\mathcal{S}(G)$ into $4$ subclasses defined as follows
(we keep notation as above).  First recall that if $S$ is a quasisimple
Chevalley group in characteristic $p$ and $V$ is an absolutely
irreducible module, then $V=V(\lambda)$ for some dominant
weight $\lambda$ (in particular, the representation extends
to the algebraic group).  Write $\lambda = \sum a_i \lambda_i$
with the $a_i$ nonnegative integers and
 the $\lambda_i$ are the fundamental weights.
A restricted representation is one with $a_ i< p$ for all $i$.
By the Steinberg tensor product theorem, every module is
a tensor product of Frobenius twists of restricted modules
(over the algebraic closure).   See \cite{Ja, St2}  for details
of this theory.

\begin{enumerate}
\item[$\mathcal{S}_1$]   $S$ is alternating or sporadic;
\item[$\mathcal{S}_2$]   $S$ is a Chevalley group in characteristic not
dividing $q$;
\item[$\mathcal{S}_3$]   $S$ is a Chevalley group in characteristic
dividing $q$ and the representation is not restricted; and
\item[$\mathcal{S}_4$]   $S$ is a Chevalley group in characteristic
dividing $q$, and the representation is restricted.
\end{enumerate}

\begin{lemma}   For $r$ sufficiently large, we have:
\begin{enumerate}
\item The number of conjugacy classes of maximal subgroups
in $\mathcal{S}_1(G)$ is at most  $ O(r^{1/2} e^{10{r^{1/4}}})  $;  and
\item $$
\frac{ |\cup_{M \in \mathcal{S}_1(G)} M|}{|G|} < O(q^{-r/3}).
$$
\end{enumerate}
\end{lemma}

\begin{proof}  We may assume that
$r$ is sufficiently large such that there are no sporadic
groups in $\mathcal{S}_1(G)$ nor alternating groups
of degree less than $17$.

Let $d$ be the dimension of the natural module for
our classical group.  So $d \le 2r+1$.

It follows  by \cite[Lemma 6]{GT}  that either the module is the natural
permutation  module for the symmetric group or  the dimension $d$ of the
module satisfies $d \ge  (m^2 -  5m + 2)/2$.  Since $d \le 2r+1$, this implies that,
for $r$ sufficiently large, aside from the natural permutation module, $m \le 3r^{1/2}$.

By the comments preceding Lemma \ref{derangements-c38}, the number of
embeddings of $A_m$ into $G$ is at most $2rk(\widehat{A_m}) \leq 4r
k(A_m)$. Thus the number of conjugacy classes of elements of
$\mathcal{S}_1(G)$ is at most
$$
(4r) \sum_{m=5}^ {3r^{1/2}}  2k(A_m) \leq 12 r^{3/2}
k(A_{\lfloor 3r^{1/2} \rfloor}).$$ Recalling from Corollary \ref{altsym}
that $k(A_m) \le k(S_m)$, and using the bound $k(S_m) \le Cm^{-1} \exp[\pi
(2m/3)^{1/2}]$ which follows from known asymptotic behavior of the
partition function (\cite{A1}, p. 70), one concludes that the number of
conjugacy classes of elements of $\mathcal{S}_1(G)$ is at most
$O(re^{5r^{1/4}})$.

Excluding the natural  representations,  the number of conjugacy classes
of each $M \in \mathcal{S}_1(G)$ is at most  $2k(A_{\lfloor 3r^{1/2}
\rfloor}) \le O(r^{-1/2}e^{5r^{1/4}})$. Thus, excluding the embedding of
$A_m$ into $G$ via the natural module, the total number of conjugacy
classes of $G$ in the union of maximal subgroups in $\mathcal{S}_1(G)$ is
at most $O(r^{1/2} e^{10r^{1/4}})$.

Finally, consider the natural embedding of $A_m$ or $S_m$ into $G$. Then
$m=d+1$ or $d+2$ (depending upon the characteristic). Moreover, since the
representation is self dual, $G$ is either symplectic or orthogonal. Thus,
there are at most $8$ conjugacy classes of such maximal subgroups, each
with at most $k(S_{2r+3})$ classes.  Thus, these maximal subgroups
contribute at most $O(r^{-1} \exp[\pi(4r/3)^{1/2}])$  conjugacy classes of
$G$.

Thus,  the total number of conjugacy classes of $G$ represented in
$\cup_{M \in \mathcal{S}_1(G)}$ is at most $O(r^{-1}\exp(4r^{1/2}))$.
Using Theorem \ref{D} which gives an upper bound for the size of a
conjugacy class gives the result.
\end{proof}

\begin{lemma}   For $r$ sufficiently large, we have:
\begin{enumerate}
\item The number of conjugacy classes of maximal subgroups
in $\mathcal{S}_2(G)$ is at most  $O(r^3)$;  and
\item $$
\frac{ |\cup_{M \in \mathcal{S}_2(G)} M|}{|G|} < O(r^4 (1 + \log_q(r))/q^{r-1}).
$$
\end{enumerate}
\end{lemma}

\begin{proof} For part 1 we argue very much as in \cite{LPS}. Let $S$ be the socle.
By the results of various authors on minimal dimensions of projective
representations (see \cite[Table 2]{T}) and Corollary \ref {B}, it follows that
there is a universal positive constant $A$ such
  that

  \begin{enumerate}
  \item[(a)]  $k(M) \le Ar$ for any $M \in \mathcal{S}_2(G)$,
  \item[(b)]   there are at most $Ar$ possibilities for $S$ up to isomorphism,  and
  \item [(c)] $k(\hat{S}) \le Ar$ for each possible $S$.
  \end{enumerate}

  By (b) and (c), the number of conjugacy classes of maximal subgroups
  of $G$ in $\mathcal{S}_2(G)$ is at most $O(r^3)$  (the extra $r$ comes
  from the possibility of equivalent representations which are not conjugate
  in $G$), and so (1) holds.

  By (a) and (1), the number of conjugacy classes in the union of all maximal
  subgroups in $\mathcal{S}_2(G)$ is at most $O(r^4)$.  Now (2) follows by this and Theorem \ref{D}.
  \end{proof}

 \begin{lemma} For $r$ sufficiently large, we have:
\begin{enumerate}
\item The number of conjugacy classes of maximal subgroups
in $\mathcal{S}_3(G)$ is at most  $O( 2r \log r)$;  and
\item $$
\frac{ |\cup_{M \in \mathcal{S}_3(G)} M|}{|G|} < O(q^{-r/3}).
$$
\end{enumerate}
\end{lemma}

\begin{proof}  Since the representation  is not restricted and $M$ is maximal,
it  follows by Steinberg's tensor product theorem
 that the representation  must be the tensor  product of Frobenius twists
 of some restricted representation.   By \cite[Lemma 26]{GT2},
 $N_G(M)$ also preserves this tensor product (over the algebraic closure).

 Since $M$ is maximal,  this implies
 that $M$ is a classical group over a larger field and
 that $V$ is the tensor  product of Frobenius twists of the natural module
 for the classical group (and so this module is defined  over the smaller
 field).    Thus, there  will be at  most $2r \log r$ choices for  the class
 of $M$ (essentially depending upon writing
 the dimension as power of a positive integer).

 Indeed, let $d$ be the dimension of the natural module.  Write $d=m^e$
 with $e > 1$.   Then the socle of $M$ is a classical group
 over the field of size $q^e$ and of rank less than $m$.   Thus,
 by Corollary \ref{B}, $k(M) \le O(q^{em}) \le O(q^{3r^{1/2}})$.

It follows that $\cup_{M \in \mathcal{S}_3(G)} M$ contains at most
$O(r \log r q^{3r^{1/2}})$ conjugacy classes of $G$. Now apply Theorem \ref{D} to
conclude that (2) holds.
 \end{proof}

  \begin{lemma} For $r$ sufficiently large with
  $p$ the characteristic of $G$,  we have:
\begin{enumerate}
\item The number of conjugacy classes of maximal subgroups
in $\mathcal{S}_4(G)$ is at most  $O( r^{3/2}  p^{3r^{1/2}})$;  and
\item $$
\frac{ |\cup_{M \in \mathcal{S}_4(G)}  M|}{|G|} < O(q^{-r/2}).
$$
\end{enumerate}
\end{lemma}

\begin{proof}   The restricted representations of the groups
of Ree type (i.e.  one of ${^2}B_2, {^2}G_2, {^2}F_4$)
have bounded dimension and so we may ignore these.

If $S$ is an untwisted  Chevalley group over the field of
$s$ elements, then every restricted
representation is defined over that field and so $s$ must be the field
for the natural module for $G$.
If $S$ is twisted, then any representation is defined over the field of $s$
elements or $s^d$ elements where $d \le 3$  (depending upon the twist).

It also follows by \cite{Lu1} that the rank of $S$ is at most $3r^{1/2}$.

   Thus
there are most $Dr^{1/2}$ choices for $S$ for an absolute constant $D$ and
the number of possible representations is at most  $p^{3r^{1/2}}$. Thus,
the the number of possible conjugacy classes is at most
$O(r^{3/2}p^{3r^{1/2}})$; the extra $r$ comes from the possibility of
equivalent representations not conjugate in $G$. This proves (1).

By Corollary \ref{B} and the remarks above,  $k(M) \le O(q^{3r^{1/2}})$
for  each possible  $M$.

Thus, $\cup_{M \in \mathcal{S}_4(G)}  M$ is the union of at most
$O(r^{3/2}q^{3r^{1/2}}p^{3r^{1/2}})$ conjugacy classes of $G$. Now apply
Theorem \ref{D} to conclude that (2) holds.
\end{proof}

Putting the previous results together completes the proof
of Theorem \ref{stronger}.
Theorem \ref{E} also  follows immediately from the previous results.

\end{document}